\documentclass[letterpaper,reqno,11pt,oneside]{amsart}

\usepackage{amsmath,amsthm,amsfonts,amssymb}
\usepackage{enumitem,comment}
\usepackage[usenames,dvipsnames]{color}
\usepackage{mathtools}
\usepackage[extension=pdf]{hyperref}
\usepackage[totalheight=9.6in,totalwidth=5.95in,centering]{geometry}
\usepackage{graphics,graphicx}
\usepackage[color,notref,notcite,final]{showkeys}
\usepackage[style=alphabetic,natbib=true,maxnames=99,isbn=false,doi=false,url=false,firstinits=true,hyperref=auto,arxiv=abs,backend=bibtex]{biblatex}
\DeclareFieldFormat[article,inbook,incollection,inproceedings,patent,thesis,unpublished]{title}{#1\isdot}
\renewbibmacro{in:}{%
  \ifentrytype{article}{}{%
    \printtext{\bibstring{in}\intitlepunct}}} 
\defbibheading{apa}[\refname]{\section*{#1}}
\DeclareFieldFormat{sentencecase}{\MakeSentenceCase{#1}} \renewbibmacro*{title}{%
  \ifthenelse{\iffieldundef{title}\AND\iffieldundef{subtitle}}{}
  {\ifthenelse{\ifentrytype{article}\OR\ifentrytype{inbook}%
      \OR\ifentrytype{incollection}\OR\ifentrytype{inproceedings}%
      \OR\ifentrytype{inreference}} {\printtext[title]{%
        \printfield[sentencecase]{title}%
        \setunit{\subtitlepunct}%
        \printfield[sentencecase]{subtitle}}}%
    {\printtext[title]{%
        \printfield[titlecase]{title}%
        \setunit{\subtitlepunct}%
        \printfield[titlecase]{subtitle}}}%
    \newunit}%
  \printfield{titleaddon}}
\AtEveryBibitem{\clearlist{language}}

\definecolor{darkblue}{rgb}{0.13,0.13,0.39}
\hypersetup{colorlinks=true,urlcolor=darkblue,citecolor=darkblue,linkcolor=darkblue,%
  pdftitle=Non-intersecting Brownian bridges and LOE}

\mathtoolsset{showmanualtags,showonlyrefs}

\newtheorem{thm}{Theorem}[section] 
\newtheorem{lem}[thm]{Lemma}
 
\newtheorem{prop}[thm]{Proposition} 
 
\theoremstyle{definition}  \newtheorem*{rem*}{Remark}
 
 \newcounter{assum}

\newcommand{\I}{{\rm i}} \newcommand{\pp}{\mathbb{P}} 
  \newcommand{\rr}{\mathbb{R}}
\newcommand{\nn}{\mathbb{N}} \newcommand{\zz}{\mathbb{Z}} \newcommand{\aip}{\mathcal{A}_2}

\newcommand{\p}{\partial}
\newcommand{\uno}[1]{\mathbf{1}_{#1}}

\newcommand{\vs}{\vspace{6pt}}
\newcommand{\wt}{\widetilde}

\newcommand{\K}{{\sf K}_{\Ai}}  
\newcommand{\qand}{\quad\text{and}\quad}
\newcommand{\qqand}{\qquad\text{and}\qquad}
\newcommand{\ts}{\hspace{0.1em}}
\newcommand{\tsm}{\hspace{-0.1em}}
\newcommand{\ttsm}{\hspace{-0.05em}}

\DeclareMathOperator{\Ai}{Ai}
\DeclareMathOperator{\tr}{tr}

\newcommand{\inv}[1]{\frac{1}{#1}}

\newcommand{\KGN}{{\sf K}_{{\rm Herm},N}}
\newcommand{\wtKGN}{{\wt {\sf K}}_{{\rm Herm},N}}

\newcommand{\KLN}{{\sf K}_{{\rm Lague},N}}
\newcommand{\wtKLN}{{\wt {\sf K}}_{{\rm Lague},N}}

\DeclareMathOperator*{\argmax}{argmax}

\def\dash---{\kern.16667em---\penalty\exhyphenpenalty\hskip.16667em\relax}

\numberwithin{equation}{section}

\let\oldmarginpar\marginpar
\renewcommand\marginpar[1]{\-\oldmarginpar[\raggedleft\footnotesize #1]%
  {\raggedright{\small\textsf{#1}}}}

\allowdisplaybreaks[2]

\begin{document}

\title{Non-intersecting Brownian bridges and the Laguerre Orthogonal Ensemble}

\author[G.~B.~Nguyen]{Gia Bao Nguyen} 
\address{G.~B.~Nguyen\\
  Centro de Modelamiento Matem\'atico\\
  Universidad de Chile\\
  Av. Beauchef 851, Torre Norte\\
  Santiago\\
  Chile} \email{bnguyen@dim.uchile.cl}

\author[D.~Remenik]{Daniel Remenik}
\address{D.~Remenik\\
  Departamento de Ingenier\'ia Matem\'atica and Centro de Modelamiento Matem\'atico\\
  Universidad de Chile\\
  Av. Beauchef 851, Torre Norte\\
  Santiago\\
  Chile} \email{dremenik@dim.uchile.cl}

\maketitle

\begin{abstract}
  We show that the squared maximal height of the top path among $N$ non-intersecting Brownian bridges starting and ending at the origin is distributed as the top eigenvalue of a random matrix drawn from the Laguerre Orthogonal Ensemble. 
  This result can be thought of as a discrete version of K. Johansson's result that the supremum of the Airy$_2$ process minus a parabola has the Tracy-Widom GOE distribution, and as such it provides an explanation for how this distribution arises in models belonging to the KPZ universality class with flat initial data.
  The result can be recast in terms of the probability that the top curve of the stationary Dyson Brownian motion hits an hyperbolic cosine barrier.
\end{abstract}

\section{Introduction and Main Results}\label{sec:intro}

\subsection{Motivation and background}\label{sub:motiv}

The \emph{Kardar-Parisi-Zhang (KPZ) universality class} describes a broad collection of models, including stochastic interface growth on a one-dimensional substrate, polymer chains directed in one dimension and fluctuating transversally in the other due to a random potential, driven lattice gas models,
reaction-diffusion models in two-dimensional random media, and randomly forced Hamilton-Jacobi equations.
Although there is no precise definition of the KPZ universality class, it can be identified at the roughest level by its unusual $t^{1/3}$ scale of fluctuations (decorrelating at a spatial scale of $t^{2/3}$).
The asymptotic distribution of the fluctuations, in the long time limit $t\to \infty$, is conjectured to depend only on the initial (or boundary) condition imposed on each particular model.

There are three special classes of initial data which stand out because of their scale invariance, usually referred to as \emph{curved}, \emph{flat} and \emph{stationary}.
Based on exact computations for a few models which enjoy a special determinantal structure, the distribution of the asymptotic fluctuations in these three cases is known explicity.
One of the most intriguing aspects of the KPZ universality class is that these limiting fluctuations are given in terms of objects coming from random matrix theory (RMT).
This is particularly evident in the cases of curved and flat initial data: the asymptotic fluctuations are given, respectively, by the Tracy-Widom GUE and GOE distributions \cite{tracyWidom,tracyWidom2}.
The first of these two distributions describes the asymptotic fluctuations of the largest eigenvalue of a random Hermitian matrix with Gaussian entries (the Gaussian Unitary Ensemble), while the second one is the analog in the real symmetric case (the Gaussian Orthogonal Ensemble); both will be introduced explicitly later on.
For more background on this aspect of the KPZ universality class we refer the reader to the reviews \cite{corwinReview,quastelRem-review}; for some other perspectives we refer additionally to \cite{quastelCDM,borodinPetrovReview,quastelSpohn}.

It is very natural in this context to wonder about what lies behind the connection between the KPZ class and RMT.
Perhaps the most basic relationship one may seek is to find a model which lies in the KPZ universality class and which, at the same time, is naturally expressed as an object in RMT.
As it turns out, in the case of the GUE (corresponding to curved initial data in the KPZ class) this can be achieved by considering a simple model: non-intersecting Brownian bridges (which we will introduce in detail in Section \ref{sec:nibm}).
This model is, on the one hand, one of the simplest and most studied models belonging to the KPZ class, while on the other hand it is equivalent to Dyson Brownian motion, a process which describes the evolution of the eigenvalues of a GUE matrix whose entries undergo independent Ornstein-Uhlenbeck diffusions.
A straightforward consequence of this equivalence is that the positions of the $N$ non-intersecting Brownian paths at a single time are distributed as the eigenvalues of a GUE matrix of size $N$, and this leads directly to analog statements about their asymptotic fluctuations.
Interestingly, the scope of this relationship extends also to looking at the entire paths of these processes.
For instance, if one scales the top path of Dyson Brownian motion (or non-intersecting Brownian bridges) appropriately, then in the limit one obtains the Airy$_2$ process, which is known to describe the spatial fluctuations of models in the KPZ class with curved initial data.
Beyond the basic relationship which we have just described, more recent developments in the area known as integrable probability have led to other, arguably deeper, ways of understanding the connection (see for instance \cite{borodinPetrovReview,borodinGorinBetaCorners}).

The situation in the case of GOE, which corresponds to flat initial data in the KPZ class, is much less clear.
In fact, essentially no results are available, and it has been a question of interest for several years now, both for probabilists and for physicists, to understand whether a relationship similar to the one available for the GUE case is available for GOE, or whether the appearance of the Tracy-Widom GOE distribution in the KPZ class is not much more than a coincidence.

The fact that the GOE/flat link is much more difficult to understand is actually not surprising given that, as it is widely accepted, for most (if not all) models both in the KPZ class and in RMT, the GOE/flat case is considerably more difficult to analyze than the GUE/curved one.
This is because many aspects of the integrability of these models which are present in the latter case, and lead to relatively simple exact formulas, are lost in the former one.
It should be noted moreover that, in a certain sense, the GOE/flat connection is necessarily more tenuous than the GUE/curved one.
In fact, if one considers the GOE version of Dyson Brownian motion then it is natural to expect (as conjectured in \cite{borFerPrahSasam}), by analogy with the GUE case, that the top path would converge, under appropriate scaling, to the Airy$_1$ process, which is the analog of the Airy$_2$ process for models in the KPZ class with flat initial data.
Nevertheless, \cite{bornemannFerrariSpohnAiry1} provided convincing numerical evidence showing that this is not the case.

The main goal of this article is to provide an explanation of how the GOE/flat link arises.
We will achieve this by considering the model of non-intersecting Brownian bridges mentioned above but focusing now on a different quantity, namely the maximal height attained by the top path.
Our main result will show that the distribution of the maximal height coincides with that of the largest singular value of a large rectangular matrix with Gaussian entries, or in other words, with the square root of the largest eigenvalue of a matrix from the Laguerre Orthogonal Ensemble, i.e. a real Wishart matrix.
We remark that this identity will be established at the pre-asymptotic level (that is, for a finite number of paths and for a finite matrix), which is interesting in itself as we will explain in Section \ref{sub:the_airy_2_process_and_goe}.
The connection with the GOE is established through the known RMT fact that the square root of the top eigenvalue of a real Wishart matrix converges under the right scaling to a Tracy-Widom GOE random variable.
The way in which this result fits into the context of the KPZ universality class with flat initial data can be understood in terms of certain variational problems, and will be explained in Section \ref{sub:main_results}.

In the next two subsections we will change a bit our perspective to focus in more detail on the model of non-intersecting Brownian bridges, as well as on the Airy$_2$ process and on some previous results which relate it to the Tracy-Widom GOE distribution.

\subsection{Non-intersecting Brownian bridges}\label{sec:nibm}

The model of \emph{non-intersecting Brownian bridges} corresponds to considering a collection of $N$ Brownian bridges $(B_1(t),\dotsc,B_N(t))$, all starting from zero at time $t=0$ and ending at zero at time $t=1$, and conditioning them (in the sense of Doob) to not intersect in the region $t\in(0,1)$.
We will always assume that the paths are ordered so that $B_1(t)<\dotsm<B_N(t)$ for $t\in(0,1)$.
This model (which in the physics and combinatorics literatures is sometimes referred to as \emph{watermelons without a wall}) and variants of it have been studied extensively in the last decade, see for instance \cite{tracyWidomDysonBM,adlerVanMoerbeke-PDEs,baikSuidan-nirw,feierl,kobIzumKat,delvauxKuijlaarsZhang,ferrariVeto-Tacnode,liechty-nibmdope,johansson-BMtacnode} among many others. 
The model can be thought of as a limit of non-intersecting random walks, which in the physics literature are known as \emph{vicious walkers}, and were introduced by \citet{fisher} (under an additional conditioning on the walks staying positive) as a model for wetting and melting.

The interest in studying systems of non-intersecting paths, both in the statistical physics and probability literatures, is due in large part to their intimate connection with RMT and the KPZ universality class.
As an example, it has been shown that as the number of paths $N\to\infty$, and under proper scaling, several variants of these models converge to \emph{universal} processes, such as the sine, Airy, Pearcey and tacnode processes.
Universal here means that the same limiting processes arise for a wide class of other models (for more on this aspect see \cite{johansson-BMtacnode,adlerFerrarivanMoerbeke} and references therein).
The first two of these universal processes also arise naturally in RMT.
For instance, and as we already mentioned, for fixed $t\in(0,1)$ the distribution of $(B_1(t),\dotsc,B_N(t))$ coincides (modulo some scaling) with that of the eigenvalues of a random matrix drawn from the Gaussian Unitary Ensemble (GUE) and converge, under suitable scaling at the edge of the GUE spectrum, to the Airy point process.

A particular aspect which has been subject of intense research has been the study of the maximal height attained by the highest path of a collection of non-intersecting paths.
In the physics literature, \cite{SMCR,feierl2,RS1,RS2} obtain various expressions for the distribution of this maximum.
As in the case of the limiting processes mentioned above, their main motivation lies in the computation of the asymptotic distribution in the $N\to\infty$ limit, which for many different models is conjectured to be given by the Tracy-Widom GOE distribution.
This was achieved in the physics literature using non-rigorous methods (see for instance \cite{forrester}, which further establishes connections with Yang-Mills theory).
For the case of non-intersecting Brownian motions on the half-line (with either absorbing or reflecting boundary condition at zero) this was rigorously proved by \citet{liechty-nibmdope}.

In this paper we will focus on the distribution of the maximal height of a \emph{finite} number of non-intersecting Brownian bridges.
More precisely, for fixed $N$ we are interested in the distribution of the random variable
\begin{equation}\label{eq:MN}
\mathcal{M}_N=\max_{t\in[0,1]}B_N(t).
\end{equation}
As we already mentioned, under proper centering and scaling $\mathcal{M}_N$ should converge in distribution as $N\to\infty$ to a Tracy-Widom GOE random variable.
The question in which we will be interested here is whether there is a finite $N$ version of this result.
Rather surprisingly, and as we mentioned already above, we will find that the answer is yes.
But before stating the result, and in order to provide additional motivation (and in particular explain why this is in itself a natural question), let us discuss in some detail the GOE result in the $N\to\infty$ regime.

\subsection{The Airy$_2$ process and GOE} 
\label{sub:the_airy_2_process_and_goe}

The Airy${}_2$ process $\aip$ was introduced by \citet{prahoferSpohn} in the study of the
scaling limit of a discrete polynuclear growth (PNG) model.
It is expected to govern the asymptotic spatial fluctuations in a wide variety of random growth models on a one-dimensional substrate with curved initial conditions, and the point-to-point free energies of directed random polymers in $1+1$ dimensions.
For its definition and a detailed discussion of its properties and relevance we refer the reader to \cite{quastelRem-review}; let us just mention that the Airy$_2$ process is non-Markovian and stationary, with marginal distributions given by the Tracy-Widom GUE distribution.

The Airy$_2$ process is also known to arise in the setting of \emph{(geometric) last passage percolation}.
Here one considers a family $\big\{w(i,j)\}_{i,j\in\zz^+}$ of independent geometric random variables with parameter
$q$ (i.e. $\pp(w(i,j)=k)=q(1-q)^{k}$ for $k\geq0$) and let $\Pi_N$ be the collection of up-right paths of length $N$, that is, paths $\pi=(\pi_0,\dotsc,\pi_n)$ such that $\pi_i-\pi_{i-1}\in\{(1,0),(0,1)\}$.
The \emph{point-to-point last passage time} is defined, for $M,N\in\zz^+$, by 
\[L^{\rm point}(M,N)=\max_{\pi\in\Pi_{M+N}:(0,0)\to(M,N)}\sum_{i=0}^{M+N}w(\pi_i),\]
where the maximum is taken over all up-right paths connecting the origin to $(M,N)$.
\citet{johanssonShape} proved that there are explicit constants $c_1$ and $c_2$, depending only on $q$, such that $\pp\big(L^{\rm point}(N,N)\leq c_1N+c_2N^{1/3}r\big)\longrightarrow F_{\rm GUE}(r)$ as $N\to\infty$, with $F_{\rm GUE}$ the Tracy-Widom GUE distribution.
Next one defines a process $t\mapsto H_N(t)$ by linearly interpolating the values given by scaling $L^{\rm point}(N,M)$ through the relation
\begin{equation}\label{eq:lpp-pt-scaling}
L^{\rm point}(N+k,N-k)=c_1N+c_2N^{1/3}H_n(c_3N^{-2/3}k),
\end{equation}
where $c_3$ is another explicit constant which depends only on $q$.
\citet{johansson} went on to show that
\begin{equation}\label{eq:johAiry2}
  H_N(t) \longrightarrow \aip(t)-t^2
\end{equation}
in distribution, in the topology of uniform convergence on compact sets.
On the other hand one can define the \emph{point-to-line last passage time} by
\begin{equation}
L^{\rm line}(N)=\max_{k=-N,\dots,N}L^{\rm point}(N+k,N-k).\label{eq:lpplptline}
\end{equation}
From the definition and Johansson's result \eqref{eq:johAiry2} it follows that
\[c_2^{-1}N^{-1/3}[L^{\rm line}(N)-c_1N]\longrightarrow\sup_{t\in\rr}\{\aip(t)-t^2\}\]
in distribution.
But it was known separately \cite{baikRains} that the quantity on the left converges in distribution to a Tracy-Widom GOE random variable, from which Johansson deduced in \cite{johansson} the remarkable fact that
\begin{equation}\label{eq:johGOE}
  \pp\left(\max_{t\in\rr}\,(\aip(t)-t^2)\leq r\right)=F_{\rm GOE}(4^{1/3}r),
\end{equation}
where $F_{\rm GOE}$ denotes the Tracy-Widom GOE distribution (an explicit formula for $F_{\rm GOE}$ will be given in Section \ref{sub:goe_and_loe}).
A more direct proof of \eqref{eq:johGOE} was given in \cite{cqr}, based on formulas for the hitting probabilities for the Airy$_2$ process.
This method has led to several other results about the Airy$_2$ and related processes (see e.g. \cite{mqr} or the review \cite{quastelRem-review}) and it is the one we will use in this paper in the context of non-intersecting Brownian bridges.

The relation between the Airy$_2$ process and the study of $\mathcal{M}_N$ lies in the fact that, suitably rescaled, the top path of a collection of non-intersecting Brownian bridges converges to the Airy$_2$ process minus a parabola: 
\begin{equation}\label{eq:bbAiry2}
  2N^{1/6}\Big(B_N\big(\tfrac12(1+N^{-1/3}t)\big)-\sqrt{N}\Big)\longrightarrow\aip(t)-t^2
\end{equation}
in the sense of convergence in distribution in the topology of uniform convergence on compact sets.
This result is well-known in the sense of convergence of finite-dimensional distributions; the stronger convergence stated here was proved in \cite{corwinHammond}.
In view of this result, a similar argument as the one leading to \eqref{eq:johGOE} together with \eqref{eq:johGOE} itself gives the following:

\begin{thm}\label{thm:bbGOE}
  \begin{equation}\label{eq:bbGOE}
    \lim_{N\to\infty}\pp\Big(2N^{1/6}\big(\max_{t\in[0,1]}B_N(t)-\sqrt{N}\big)\leq r\Big)=F_{\rm GOE}(4^{1/3}r).
  \end{equation}
\end{thm}

It is this version of Johansson's result \eqref{eq:johGOE} which provided the original motivation for our paper.
We remark that, as a by-product of our results, we obtain a more direct derivation of \eqref{eq:bbGOE}.


\subsection{GOE and LOE} 
\label{sub:goe_and_loe}

In this section we will quickly introduce the two ensembles of random matrices which are most relevant to our results.
The first one is the \emph{Gaussian Orthogonal Ensemble (GOE)}. 
Let $\mathcal{N}(a,b)$ denote a Gaussian random variable with mean $a$ and variance $b$.
An $N\times N$ \emph{GOE matrix} is a symmetric matrix $A$ such that
$A_{ij}=\mathcal{N}\hspace{-0.1em}(0,1)$ for $i>j$ and $A_{ii}=\mathcal{N}\hspace{-0.1em}(0,2)$, where all the Gaussian variables are independent (subject to the symmetry condition).
The term orthogonal refers to the fact that the distribution of a GOE matrix is invariant under conjugation by an orthogonal matrix.
The GOE can be regarded as the probability measure on the space of $N\times N$ real symmetric matrices $A$ with density
$\frac{1}{C_N}e^{-\frac1{4}\tr(A^2)}$ for some normalization constant $C_N$.
The joint density of the eigenvalues $(\lambda_1,\dotsc,\lambda_N)$ of a GOE matrix can be explicitly computed, and is given by
\[\frac{1}{Z_N}\prod_{i=1}^Ne^{-\frac{1}{4}\lambda^2_i}\prod_{1\leq i<j\leq N}|\lambda_i - \lambda_j|\]
for some other normalization constant $Z_N$.
The weights $e^{-\lambda_i^2/4}$ appearing in this formula are the weights associated to the Hermite polynomials in the theory of orthogonal polynomials; for this reason, the Gaussian ensembles such as the GOE are sometimes also referred to as \emph{Hermite ensembles}.
The Wigner semicircle law \cite{wigner} states that the empirical eigenvalue density for the GOE has approximately a semicircle
distribution on the interval $[-2\sqrt{N},2\sqrt{N}]$.
The fluctuations of the spectrum at its edges are of order $N^{-1/6}$ and give rise to the \emph{Tracy-Widom GOE distribution}: denoting by $\lambda_{\rm GOE}(N)$ the largest eigenvalue of an $N\times N$ GOE matrix, we have \cite{tracyWidom2}
\begin{equation}\label{eq:GOElim}
\lim_{N\to\infty}\pp\big(\lambda_{\rm GOE}(N)\leq2\sqrt{N}+N^{-1/6}r\big)=F_{\rm GOE}(r)
\end{equation}
with
\begin{equation}\label{eq:GOE}
F_{\rm GOE}(r)=\det({\sf I}-{\sf P}_0{\sf B}_r{\sf P}_0)_{L^2(\rr)},
\end{equation}
where ${\sf P}_r$ denotes the projection onto the interval $(r,\infty)$ (i.e. ${\sf P}_rf(x)=f(x)\uno{x>r}$ for $f\in L^2(\rr)$), ${\sf B}_r$ is the integral operator acting on $L^2(\rr)$ with kernel
\begin{equation}
  {\sf B}_r(x,y)=\Ai(x+y+r),\label{eq:defB0}
\end{equation}
and $\Ai$ denotes the Airy function.
The determinant in \eqref{eq:GOE} means the Fredholm determinant on the Hilbert space $L^2(\rr)$.
For the definition, properties and some background on Fredholm determinants, which can be thought of as the natural generalization of the usual determinant to infinite dimensional Hilbert spaces, we refer the reader to \cite[Section 2]{quastelRem-review}.
We remark that \eqref{eq:GOE} is not the original formula provided in \cite{tracyWidom2} (which is written in terms of Painlev\'e II transcendents instead of Fredholm determinants); this formula is essentially due to \cite{sasamoto}, and was proved in \cite{ferrariSpohn}.
Note also that one can choose a slightly different scaling (with the entries of a GOE matrix having variances $N$ off the diagonal and $2N$ on the diagonal) so that the edge of the spectrum is at $2N$ and the fluctuations are of order $N^{1/3}$, which leads to a scaling in \eqref{eq:GOElim} of the same order as that in \eqref{eq:lpp-pt-scaling}.

We turn now to the \emph{Laguerre Orthogonal Ensemble (LOE)}. Let $X$ be an $n\times N$ matrix whose entries are i.i.d. $\mathcal{N}\hspace{-0.1em}(0,1)$, where we assume $n\geq N$.
Then the matrix $M=X^{\sf T}\tsm X$ is said to be an $N\times N$ LOE matrix (often referred to also as a \emph{Wishart matrix}). 
In applications to statistics, one thinks of the rows of $X$ as containing $n$ independent samples of an $N$-variate Gaussian population (with covariance matrix given by the identity), so that $\frac{1}{n}M$ corresponds to the sample covariance matrix.
The joint density of the eigenvalues of $M$ is also explicit in this case, and is given by
\[\frac{1}{Z_N}\prod_{1\leq i<j\leq N}|\lambda_i - \lambda_j|\prod_{i=1}^N \lambda_i^{a}\ts e^{-\lambda_i/2},\]
where the parameter $a$ is defined to be $a=(n-N-1)/2$.
The weights $\lambda_i^{a}e^{-\lambda_i/2}$ appearing in this case are the ones associated to the (generalized) Laguerre polynomials, which explains the name of this family of random matrices.
By the Mar\v{c}enko-Pastur law \cite{marcenkoPastur} the eigenvalues of $M$ are concentrated on the interval $[0,4N]$.
Under our scaling, if $a=(n-N-1)/2$ is fixed (and independent of $N$) then the fluctuations at the \emph{soft edge} $4N$ are of order $N^{1/3}$, and have the same limiting distribution as in the GOE case \cite{johnstone}: denoting by $\lambda_{\rm LOE}(N)$ the largest eigenvalue of the LOE matrix, we have
\begin{equation}\label{eq:LOElim}
  \lim_{N\to\infty}\pp\big(\lambda_{\rm LOE}(N)\leq4N+2^{4/3}N^{1/3}r\big)=F_{\rm GOE}(r).
\end{equation}
The scaling at the \emph{hard edge} at the origin is different and gives rise to a different limit distribution, but we will not need it in this paper.

In all that follows we will be interested exclusively in the case $a=0$.


\subsection{Main results} 
\label{sub:main_results}

We are ready now to state the main result of this paper.
Let $M$ be an LOE matrix with $a=0$, that is, $M=X^{\sf T}\tsm X$ with $X$ an $(N+1)\times N$ matrix with independent $\mathcal{N}\hspace{-0.1em}(0,1)$ entries.
For this choice of $a$ we will denote by $F_{{\rm LOE},N}$ the distribution function of the largest eigenvalue of $M$,
\begin{equation}\label{eq:def-FLOEN}
  F_{{\rm LOE},N}(r) = \pp\big(\lambda_{\rm LOE}(N)\leq r).
\end{equation}
Recall the definition in \eqref{eq:MN} of $\mathcal{M}_N$ as the maximum height of a collection of $N$ non-intersecting Brownian bridges. 

\begin{thm}\label{thm:bbLOE}
  Let $B_1(t)<\dotsm<B_N(t)$ be a collection of non-intersecting Brownian bridges as above. Then for all $r\geq0$ we have
  \begin{equation}\label{eq:bbLOE}
  \pp\!\left(\max_{t\in[0,1]}\sqrt{2}B_N(t)\leq r\right) = F_{{\rm LOE},N}(2r^2).
  \end{equation}
  In other words, $4\mathcal{M}_N^{2}$ is distributed as the largest eigenvalue of an LOE matrix or, alternatively, $2\mathcal{M}_N$ is distributed as the largest singular value of the $(N+1)\times N$ matrix $X$ introduced above.
\end{thm}

Let us quickly verify that the scaling in this result is consistent with the one in Theorem \ref{thm:bbGOE} and \eqref{eq:LOElim}.
Theorem \ref{thm:bbLOE} says that $\mathcal{M}_N\,{\buildrel ({\rm d}) \over =}\,\sqrt{\lambda_{\rm LOE}(N)/4}$.
By \eqref{eq:LOElim}, this implies that
\[\mathcal{M}_N = \sqrt{N+2^{-2/3}N^{1/3}\zeta_{\rm GOE}+o(N^{1/3})}
=\sqrt{N}+2^{-5/3}N^{-1/6}\zeta_{\rm GOE}+o(N^{-1/6}),\]
where $\zeta_{\rm GOE}$ is a Tracy-Widom GOE random variable.
In other words, $2N^{1/6}\big(\mathcal{M}_N-\sqrt{N}\big)=4^{-1/3}\zeta_{\rm GOE}+o(1)$, which is exactly the content of Theorem \ref{thm:bbGOE}.
In particular, Theorem \ref{thm:bbGOE} follows as a corollary of \eqref{eq:LOElim} and \eqref{eq:bbLOE}.

We take a brief pause now and go back to an issue left open at the end of Section \ref{sub:motiv}, which is the question of why $\mathcal{M}_N$ should be interpreted as a flat initial data object in the KPZ universality class.
In a way, the convergence of $\mathcal{M}_N$ a Tracy-Widom GOE random variable should be taken, in itself, as enough evidence of this fact.
But the connection goes a bit further, and can be understood in terms of certain variational formulas.
For example, in the context of last passage percolation (LPP), the point-to-line last passage times \eqref{eq:lpplptline} leading to GOE fluctuations are defined in terms of the maximum of point-to-point last passage times \eqref{eq:lpp-pt-scaling}, which in turn lead to GUE fluctuations.
The parallel with \eqref{eq:bbLOE} is direct.
The exact same relationship can be established at the level of many other polymer models (of which LPP is a zero-temperature version), and at the level of the totally asymmetric exclusion process (which can be mapped to LPP).

This straightforward way of expressing flat initial data quantities in terms of their curved initial data analogues is not as explicit in the case of some other models, such as the partially asymmetric exclusion process, which have less (or at least a more complicated) algebraic structure, but it is interesting to note that it does hold at the level of another of the most important members of the KPZ universality class, the KPZ equation.
Without going into much detail, the KPZ equation can be understood by studying the stochastic heat equation (SHE), which is linear.
It turns out that the flat initial data for the KPZ equation corresponds to starting the SHE with initial condition $Z_0\equiv1$, and thus by linearity the flat solution can be obtained by convolving the constant function $1$ with the solution of the SHE starting with $Z_0=\delta_0$, which corresponds to curved initial data.
Note that the relationship in this last case is not written directly in terms of a variational problem as described before, but one can check that (at least conjecturally, by essentially appealing to a version of Laplace's method) one recovers a variational problem as time $t\to \infty$.
For much more on this see \cite{quastelRem-review,cqrKPZfxpt,quastelRemHowFlat}.

Coming back to the description of our main results, Theorem \ref{thm:bbLOE} is equivalent to a statement about the probability that the top path of Dyson Brownian motion hits an hyperbolic cosine barrier, and it is that version of the result which we will prove.
Consider an $N\times N$ random matrix drawn from the \emph{Gaussian Unitary Ensemble}, that is, a (complex-valued)
Hermitian matrix $A$ such that $A_{ij}=\mathcal{N}\tsm(0,1/4)+\I\ts\mathcal{N}\tsm(0,1/4)$ for $i>j$ and $A_{ii}=\mathcal{N}\tsm(0,1/2)$, where all the Gaussian variables are independent (subject to the Hermitian condition).
Note that (for later convenience) we have chosen a somewhat different scaling for the Gaussian variables here compared with our definition of the GOE matrices.
Now suppose that we let $A$ evolve by letting each Gaussian variable in the construction diffuse according to independent copies of the Ornstein-Uhlenbeck process $X_t$ defined as the solution of the SDE
\[dX_t=-X_t\ts dt+\sigma\ts dW_t,\]
where $W_t$ is a standard Brownian motion and $\sigma=\frac1{\sqrt2}$ for off-diagonal entries and $\sigma=1$ on the diagonal.
We write the eigenvalues of this matrix at time $t$ as $(\lambda_1(t),\dotsc,\lambda_N(t))$, with $\lambda_i(t)$ increasing with $i$.
This eigenvalue diffusion is known as the \emph{stationary (GUE) Dyson Brownian motion}, and it defines an ensemble of almost surely non-intersecting curves indexed by $\rr$.
It can alternatively be written as the solution of a certain $N$-dimensional SDE, and it is not hard to check that it is stationary, with marginals at any time $t$ given by the eigenvalue distribution of an $N\times N$ GUE matrix.

\begin{thm}\label{thm:dbmLOE}
Let $(\lambda_1(t),\dotsc,\lambda_N(t))$ be the stationary Dyson Brownian motion defined above and let $F_{{\rm LOE},N}$ be defined as in \eqref{eq:def-FLOEN}. Then
\begin{equation}\label{eq:dbmLOE}
  \pp(\lambda_N(t)\leq r\ttsm\cosh(t)~\forall\,t\in\rr) = F_{{\rm LOE},N}(2r^2).
\end{equation}
\end{thm}

The equivalence between the two results is due to the fact that non-intersecting Brownian bridges can be mapped into the stationary Dyson Brownian motion in such a way that the probabilities on the left-hand side of \eqref{eq:bbLOE} and \eqref{eq:dbmLOE} coincide.
We will explain this in more detail in Section \ref{sec:dbm}.

The proof of Theorem \ref{thm:dbmLOE} has two steps.
The first one consists in obtaining an explicit formula for the probability on the left-hand side of \eqref{eq:dbmLOE}.
By the mapping between non-intersecting Brownian bridges and the stationary Dyson Brownian motion alluded to above, this is equivalent to finding a formula for the distribution of $\mathcal{M}_N$.
As we already mentioned, there are formulas in the literature for the distribution of the maximal height of several models related to non-intersecting Brownian bridges, which can be obtained through a direct application of the Karlin-McGregor/Lindstr\"om-Gessel-Viennot formula \cite{karlinMcGregor,lindstrom,gesselViennot}.
For completeness, let us state the formula in the case of $\mathcal{M}_N$ (see \cite{SMCR})\footnote{This formula was derived in \cite{SMCR} using path-integral techniques. Although we are not aware of a derivation in the literature based on the Karlin-McGregor formula, for the case of non-intersecting Brownian excursions (corresponding to imposing an absorbing boundary at zero) the analog formula, also derived in \cite{SMCR}, was rederived in this way in \cite{katoriTanemura-BMdet}.}:
\begin{equation}\label{eq:pathintegr}
  \pp(\mathcal{M}_N\leq r)=\frac{2^{2N}}{(2 \pi)^{N/2}r^{N^2}\prod_{j=1}^Nj!}\int_{[0,\infty)^N}d\vec{y}\,
e^{-\sum_jy_j^2/2r^2}\left(\hspace{-0.1em}\det\!\left[y_i^{j-1}\cos(y_i+\tfrac{j\pi}2)\right]_{i,j=1}^N\right)^2.
\end{equation}
By using the Cauchy-Binet identity, the right-hand side can be turned into a single $N\times N$ determinant with entries involving Hermite polynomials, see (102)--(103) in \cite{RS2}.
The resulting formula is reminiscent of some of the formulas we will obtain below, see \eqref{eq:KGN-H} together with \eqref{eq:dbm-formula}, but it is not clear how to use it directly to obtain a proof of Theorem \ref{thm:bbGOE} (nor of \eqref{eq:bbLOE}).
Moreover, as we will explain next, while the structure of the Fredholm determinant formula for the distribution of $\mathcal{M}_N$ which we will obtain in this paper (see Proposition \ref{prop:dbm-formula}) makes very apparent a connection with Johansson's result \mbox{\eqref{eq:johGOE}\dash---this} was an important clue for us in the discovery of \eqref{eq:bbLOE}\dash---from the formula appearing in \cite{RS2} such a connection is not at all clear.
It is worth mentioning that in the case of Brownian excursions, for which the analog of \eqref{eq:pathintegr} turns out to be slightly simpler, the analog of Theorem \ref{thm:bbGOE} (with the same limit) was proved by \citet{liechty-nibmdope} by appealing to a Riemann-Hilbert analysis of a certain system of discrete orthogonal polynomials.

Here we follow a different strategy, leading to an arguably simpler formula which also has some intrinsic interest.
Working at the level of Dyson Brownian motion, we appeal to a result of \cite{bcr} in order to obtain an expression for $\pp\!\left(\lambda_N(t)\leq r\ttsm\cosh(t)~\forall\,t\in[-L,L]\right)$, for fixed $L>0$ in terms of the Fredholm determinant of what they call a ``path-integral kernel''.
This path-integral kernel can be expressed in terms of the solution to a boundary value PDE, which we then solve explicitly.
Taking $L\to\infty$ in the resulting formula leads to the following result.
Let $\varphi_n$ be the \emph{harmonic oscillator functions} (which we will refer to as \emph{Hermite functions}), defined by $\varphi_n(x)=e^{-x^2/2}p_n(x)$, with $p_n$ the $n$-th normalized Hermite polynomial (i.e., so that $\|\varphi\|_2=1$), and define the \emph{Hermite kernel} as
\begin{equation}\label{eq:defKN}
  \KGN(x,y)=\sum_{n=0}^{N-1}\varphi_n(x)\varphi_n(y).
\end{equation}
We introduce also the \emph{reflection operator} ${\sf \varrho}_r$ on $L^2(\rr)$, defined by
\begin{equation}\label{eq:def-varrho}
  \varrho_rf(x)=f(2r-x).
\end{equation}

\begin{prop}\label{prop:dbm-formula}
  For any $r\geq 0$,
  \begin{equation}\label{eq:dbm-formula}
    \pp\!\left(\lambda_N(t)\leq r\ttsm\ttsm\cosh(t)~\forall\,t\in\rr\right) = \det\!\left({\sf I}-\KGN \varrho_r \KGN\right)_{L^2(\rr)}.
  \end{equation}
  The same formula holds for $\pp\!\left(\max_{t\in[0,1]}\sqrt{2}B_N(t) \leq r\right)$.
\end{prop}

This result will be proved in Section \ref{sec:dbm}. 

The expression on the right-hand side of \eqref{eq:dbm-formula} is a close analog of the formula for $F_{\rm GOE}$ appearing in \eqref{eq:GOE}.
To see this we introduce the \emph{Airy kernel}, defined as
\[\K(x,y)=\int_0^\infty d \lambda \Ai(x+\lambda)\Ai(y+\lambda).\]
This kernel is closely related to GUE, as it is the limiting correlation kernel of the GUE eigenvalues near the edge of the spectrum.
It is related to the Tracy-Widom GOE distribution because of the identity
  $\int_{-\infty}^\infty d \lambda \Ai(a+\lambda)\Ai(b-\lambda) = 2^{-1/3}\Ai(2^{-1/3}(a+b))$,
which (since $\K={\sf B}_0{\sf P}_0{\sf B}_0$, with ${\sf B_0}$ defined in \eqref{eq:defB0}) implies that 
\begin{equation}\label{eq:refl-Ai}
\K \varrho_r \K={\sf B}_0{\sf P}_0\wt {\sf B}_r{\sf P}_0{\sf B}_0
\end{equation}
with $\wt {\sf B}_r(x,y)=2^{-1/3}\Ai(2^{-1/3}(x+y+2r))$.
Since ${\sf B}_0^2={\sf I}$ (this identity is related to the fact that the family of functions $\big\{\!\Ai(x+\lambda)\big\}_{\lambda\in\rr}$ constitutes a generalized eigenbasis of $L^2(\rr)$), the cyclic property of the determinant and \eqref{eq:GOE} allow us to conclude that
\begin{equation}\label{eq:GOE2}
  F_{\rm GOE}(4^{1/3}r)=\det\!\left({\sf I}-\K \varrho_r \K \right) _{L^2(\rr)}.
\end{equation}
We point out that there does not appear to be a direct analog of \eqref{eq:refl-Ai} for $\KGN$ (although one can obtain explicit formulas for $\KGN \varrho_r \KGN$ involving no integrals, see for instance \eqref{eq:intherm-sumlague} and \eqref{eq:intherm}).

We can actually push the analogy between \eqref{eq:dbm-formula} and \eqref{eq:GOE2} a bit further and use it to provide a simple proof of Theorem \ref{thm:bbGOE}. 
Indeed, a simple scaling argument on the right-hand side of \eqref{eq:dbm-formula} leads to $\pp\!\left(2N^{1/6}(\mathcal{M}_N-\sqrt{N}) \leq r\right) = \det\!\left( {\sf I} - \wtKGN \varrho_r \wtKGN \right) $ with $\wtKGN(x,y) = \kappa_N\KGN(\kappa_Nx+\sqrt{2N},\kappa_Ny+\sqrt{2N})$, where $\kappa_N={2^{-1/2}N^{-1/6}}$.
On the other hand, it is well known that $\wtKGN$ converges to $\K$ as $N\to\infty$, where the convergence is strong enough to imply the convergence of the associated Fredholm determinants.
In view of \eqref{eq:GOE2}, and omitting the details, this implies Theorem \ref{thm:bbGOE}.

A related observation is that, in a sense, Proposition \ref{prop:dbm-formula} serves as a generalization of Johansson's result for the Airy$_2$ process, \eqref{eq:johGOE}.
In fact, the scaling argument used in the last paragraph leads to $\det\!\big( {\sf I} - \wtKGN \varrho_r \wtKGN \big) = \pp\big(\lambda_N(t) \leq (\kappa_Nr\tsm+\sqrt{2N})\cosh(t)~\forall\,t\in\rr\big)$.
On the other hand, it is known that $\wt \lambda_N(t)=\kappa_N^{-1}(\lambda_N(N^{-1/3}t)-\sqrt{2N})$ converges to $\aip(t)$ (this is just a restatement of \eqref{eq:bbAiry2} in view of the mapping between the two models).
If we knew that the convergence is strong enough to imply the convergence of $\pp\!\left(\wt\lambda_N(t) \leq a~\forall\,t\in\rr\right)$ with some control on $a$, then \eqref{eq:johGOE} would follow, because by the argument sketched in the last paragraph the determinant would go to $F_{\rm GOE}(4^{1/3}r)$, while $\kappa_N^{-1}\big[ (\kappa_N r+ \sqrt{2N})\cosh(N^{-1/3}t)-\sqrt{2N}\big] = r+t^2+O(N^{-2/3})$.

The second step in the proof of Theorem \ref{thm:dbmLOE} consists in showing that the right-hand side of \eqref{eq:dbm-formula}, i.e. $\det\!\left({\sf I}-\KGN \varrho_r \KGN\right)_{L^2(\rr)}$, equals $F_{{\rm LOE},N}(2r^2)$.
This is proved in Section \ref{sec:connection_with_loe}.
We remark that, together with the preceding discussion, this identity provides an alternative proof of the result of \cite{johnstone} in the case $a=0$.


\section{Hitting probabilities for Dyson Brownian motion}\label{sec:dbm}

Recall the stationary Dyson Brownian motion introduced in Section \ref{sub:main_results}.
As we mentioned, this model is intimately related to non-intersecting Brownian bridges.
The basic relation is that if one considers the non-stationary version of Dyson Brownian motion (where the Gaussian variables making up the entries of a GUE matrix evolve according to a plain Brownian motion), then the dynamics of the eigenvalues of this evolving matrix coincide with those of a collection of Brownian motions conditioned to never intersect.
The analogous relation in our setting goes through a time-change, and is given explicitly in \cite[Section 2.2.2]{tracyWidomNIBrExc}:
if $B_1(t)<\dotsm<B_N(t)$, $t\in[0,1]$, are non-intersecting Brownian bridges and $\lambda_1(t)<\dotsm<\lambda_N(t)$, $t\in\rr$, are defined as a stationary Dyson Brownian motion, then
\begin{equation}\label{eq:nibm-dbm}
  \big(B_i(t)\big)_{i=1,\dotsc,N} \,{\buildrel ({\rm d}) \over =}\,\left(\sqrt{2t(1-t)}\,\lambda_i(\tfrac12\log(t/(1-t)))\right)_{i=1,\dotsc,N}
\end{equation}
as processes defined for $t\in[0,1]$.
Changing variables $t\longmapsto e^{2s}/(1+e^{2s})$ leads to
\begin{equation}\label{eq:nibm-dbm-max}
  \max_{t\in[0,1]}B_N(t)\,{\buildrel ({\rm d}) \over =}\,\max_{t\in[0,1]}\sqrt{2t(1-t)}\,\lambda_N(\tfrac12\log(t/(1-t)))
=\sup_{s\in\rr}\frac{\lambda_N(s)}{\sqrt{2}\cosh(s)},
\end{equation}
which shows that Theorems \ref{thm:bbLOE} and \ref{thm:dbmLOE} are equivalent\footnote{A similar argument, together with the fact \cite{tracyWidomDysonBM} that $\sqrt{2}N^{1/6}(\lambda_N(N^{-1/3}t)-\sqrt{2N})$ converges to $\aip(t)$ in the sense of finite-dimensional distributions, provides a justification for a version of \eqref{eq:bbAiry2} in this weaker sense.}.
The rest of this section will thus be devoted to computing $\pp\big(\lambda_N(t)\leq r\ttsm\cosh(t)~\forall\,t\in\rr\big)$.

\subsection{Path-integral kernel} 
\label{sub:path_integral_kernel}

The finite-dimensional distributions of the stationary (GUE) Dyson Brownian motion are classically expressed through a Fredholm determinant in terms of the \emph{extended Hermite kernel} $\KGN^{\rm ext}$
\begin{equation}\label{eq:KGNext}
  \KGN^{\rm ext}(s,x;t,y)=\begin{cases}
    \sum_{n=0}^{N-1}e^{n(s-t)}\varphi_n(x)\varphi_n(y) &\mbox{if } s\geq t,\\
    -\sum_{n=N}^{\infty}e^{n(s-t)}\varphi_n(x)\varphi_n(y) &\mbox{if } s<t
  \end{cases}
\end{equation}
where $\varphi_n(x)=e^{-x^2/2}p_n(x)$ and $p_n$ is the $n$-th normalized Hermite polynomial.
Explicitly, if $-\infty<t_1<t_2<\ldots<t_n<\infty$ and $r_1,\dotsc,r_n\in\rr$, then
\begin{equation}\label{eq:dbmextkernel}
\pp\bigl(\lambda_N(t_j)\leq r_j,\,j=1,\dotsc,n\bigr)=\det\!\left({\sf I}-{\rm f}\KGN^{\rm ext}{\rm f}\right)_{L^2(\{t_1,\ldots,t_n\}\times\mathbb{R})},
\end{equation}
where we have counting measure on $\{t_1,\ldots,t_n\}$ and Lebesgue measure on $\rr$, and $\rm f$ is defined on $\{t_1,\ldots,t_n\}\times\mathbb{R}$ by
\begin{equation}\label{eq:f}
{\rm f}(t_j,x)=\uno{x\in(r_j,\infty)}
\end{equation}
(for more details see \cite{tracyWidomNIBrExc}).

The first step in our derivation is to obtain a formula for the probability that $\lambda_N(t)$ stays below $r\ttsm\cosh(t)$ on a finite interval $[-L,L]$.
To that end, we need to consider a finite mesh $t_1<\dotsm<t_n$ of $[-L,L]$, let $r_i=r\ttsm\cosh(t_i)$, and then take a limit of the corresponding probabilities as given in \eqref{eq:dbmextkernel} as the mesh size goes to zero.
But these probabilities become increasingly cumbersome as $n$ increases, due to the $n$-dependence in the $L^2$ space on which the operators act.
The way to overcome this problem is to first manipulate the right-hand side of \eqref{eq:dbmextkernel} into a Fredholm determinant of some other kernel acting on $L^2(\rr)$.
Such a formula was first stated, in the context of the Airy$_2$ process, in \cite{prahoferSpohn} (see also \cite{prolhacSpohn}), and the resulting formula was used in \cite{cqr} to obtain a formula for the probability that $\aip(t)$ stays below a given function $g(t)$ on a finite interval.
Later on, the procedure that converts the extended kernel formula into a formula with a Fredholm determinant acting on $L^2(\rr)$ was generalized in \cite{bcr} (see also \cite{quastelRemAiry1}) to a wide class of processes that includes the stationary Dyson Brownian motion, and from the resulting formula they obtained a continuum statistics formula for Dyson Brownian motion in a similar way as in \cite{cqr}.
In order to state the formula we need to introduce some operators.

First, recall the definition of the \emph{Hermite kernel} $\KGN$, given in \eqref{eq:defKN}, and note that $\KGN^{\rm ext}(t,x;t,y)=\KGN(x,y)$ for any $t$. 
Next we introduce the differential operator
\begin{equation}\label{eq:defD}
  {\sf D}=-\tfrac{1}{2}(\Delta-x^2+1)
\end{equation}
($\Delta$ is the Laplacian on $\rr$). 
${\sf D}$ and $\KGN$ are related: ${\sf D} \varphi_n=n \varphi_n$, so that $\KGN$ is the projection operator onto the space span$\{\varphi_0,\dotsc,\varphi_{N-1}\}$ associated to the first $N$ eigenvalues of ${\sf D}$.
In particular, even though $e^{t{\sf D}}$ is well-defined in general only for $t\leq 0$, $e^{t{\sf D}}\KGN$ is well defined for all $t$, and its integral kernel is given by
\begin{equation}\label{eq:etDKGN}
e^{t{\sf D}}\KGN(x,y)=\sum_{n=0}^{N-1}e^{tn}\varphi_n(x) \varphi_n(y).
\end{equation}

Now fix $\ell_1<\ell_2$ and consider a function $g\in H^1([\ell_1,\ell_2])$ (i.e. both $g$ and its derivative are in $L^2(\rr)$).
We introduce an operator $\Theta_{[\ell_1,\ell_2]}^g$ acting on $L^2(\rr)$ as follows: $\Theta_{[\ell_1,\ell_2]}^g f(x)=u(\ell_2,x)$, where $u(\ell_2,\cdot)$ is the solution at time $\ell_2$ of the boundary value problem
\begin{equation}\label{eq:bound-pde}
\begin{aligned}
\partial_t u+{\sf D}u&=0\quad\text{for}\ x<g(t),\ t\in(\ell_1,\ell_2)\\
u(\ell_1,x)&=f(x)\mathbf{1}_{x<g(\ell_1)}\\
u(t,x)&=0\quad \text{for}\ x\geq g(t).
\end{aligned}
\end{equation}

\begin{prop}[\cite{bcr}]\label{prop:dbmcont}
For any $\ell_1<\ell_2$ and $g\in H^1([\ell_1,\ell_2])$ we have
\begin{equation}\label{eq:dbmcont}
\pp\left(\lambda_N(t)<g(t)~\forall\,t\in[\ell_1,\ell_2]\right)=\det\!\left({\sf I}-\KGN+\Theta_{[\ell_1,\ell_2]}^g e^{(\ell_2-\ell_1){\sf D}}\KGN\right).
\end{equation}
\end{prop}

See \cite[Proposition 4.3 and Remark 4.4]{bcr} for more details. Here, and in the rest of this section, the Fredholm determinant is computed on the Hilbert space $L^2(\rr)$.

In order to make use of \eqref{eq:dbmcont} we need a formula for $\Theta_{[\ell_1,\ell_2]}^g$.
By the linearity of \eqref{eq:bound-pde}, $\Theta_{[\ell_1,\ell_2]}^g$ acts as an integral operator with kernel given by solving the boundary value problem with $f$ replaced by a delta function.
The next result gives a probabilistic representation for the integral kernel of $\Theta_{[\ell_1,\ell_2]}^g$.

\begin{prop}\label{prop:dbmkernel}
Let $\alpha=\frac14e^{2\ell_1}$, $\beta=\frac14e^{2\ell_2}$, and denote by $\Theta_{[\ell_1,\ell_2]}^g(x,y)$ the integral kernel of $\Theta_{[\ell_1,\ell_2]}^g$. Then
\begin{multline}\label{eq:dbmkernel}
  \Theta_{[\ell_1,\ell_2]}^g(x,y) = e^{\frac12(y^2-x^2)+\ell_2}\,
  \frac{e^{-(e^{\ell_1}x-e^{\ell_2}y)^2/(4(\beta-\alpha))}}{\sqrt{4\pi(\beta-\alpha)}}\\
  \times\mathbb{P}_{\hat{b}(\alpha)=e^{\ell_1}x,\,\hat{b}(\beta)=e^{\ell_2}y}\!\left(\hat{b}(t)\leq\sqrt{4t}\,g\bigl(\tfrac12\log(4t)\bigr)~\forall\, t\in[\alpha,\beta]\right),
\end{multline}
where the probability is computed with respect to a Brownian bridge $\hat{b}(t)$ from $e^{\ell_1}x$ at time $\alpha$ to $e^{\ell_2}y$ at time $\beta$ and with diffusion coefficient 2.
\end{prop}

\begin{proof}
  Let $u(t,x)$ be the solution to the boundary value PDE \eqref{eq:bound-pde} and consider the transformation $u(t,x)=e^{x^2/2+t}\ts v(\tau,z)$ with $\tau=\frac14e^{2t}$, $z=e^tx$.
  It is not hard to check then that $v(\tau,z)$ satisfies the following boundary value problem associated to the heat equation:
  \begin{align}
  \p_\tau v-\p^2_z v&=0\quad\text{for}\ z<\sqrt{4\tau}\,g\bigl(\log(4\tau)/2\bigr),\ \tau\in(\alpha,\beta)\\
  v(\alpha,z)&=e^{-z^2/(8\alpha)-\log(4\alpha)/2}f\bigl(z/\sqrt{4\alpha}\bigr)\uno{\{z<\sqrt{4\alpha}\,g(\log(4\alpha)/2)\}}\\
  v(\tau,z)&=0\quad \text{for}\ z>\sqrt{4\tau}\,g\bigl(\log(4\tau)/2\bigr),
  \end{align}
  where $\alpha=\frac14e^{2\ell_1}$, $\beta=\frac14e^{2\ell_2}$.
  This boundary value PDE can be solved explicitly in terms of Brownian motion by using the Feynman-Kac formula: letting $\hat b(s)$ denote a Brownian bridge with diffusion coefficient $2$, we have
  \begin{multline}
  v(\beta,z)=\int_{-\infty}^{\sqrt{4\alpha}\ts g(\log(4\alpha)/2)} dx\,e^{-x^2/(8\alpha)-\log(4\alpha)/2}f\Bigl(\frac{x}{\sqrt{4\alpha}}\Bigr)\frac{e^{-(x-z)^2/(4(\beta-\alpha))}}{\sqrt{4\pi(\beta-\alpha)}}\\
  \cdot\mathbb{P}_{\hat b(\alpha)=x,\,\hat b(\beta)=z}\left(\hat b(\tau)\leq\sqrt{4\tau}\,g\bigl(\log(4\tau)/2\bigr)\ \text{on}\ [\alpha,\beta]\right).
  \end{multline}
  Now using the fact that $u(\ell_2,y)=e^{y^2/2+\ell_2}\ts v(e^{2\ell_2}/4,e^{\ell_2}y)$ and recalling that $\alpha=\frac14e^{2\ell_1}$ we immediately obtain
  \begin{multline}
  u(\ell_2,y)=\int_{-\infty}^{e^{\ell_1}g(\ell_1)}\!dx\,e^{\frac12y^2-\frac{1}{2}e^{-2\ell_1}x^2+\ell_2-\ell_1}\,\frac{e^{-(x-e^{\ell_2}y)^2/(4(\beta-\alpha))}}{\sqrt{4\pi(\beta-\alpha)}}f(e^{-\ell_1}x)\\
  \cdot\mathbb{P}_{\hat b(\alpha)=x,\,\hat b(\beta)=e^{\ell_2}y}\left(\hat b(\tau)\leq\sqrt{4\tau}\,g\bigl(\log(4\tau)/2\bigr)\ \text{on}\ [\alpha,\beta]\right).
  \end{multline}
  Changing variables $x\mapsto e^{\ell_1}x$ in the integral, the formula for $\Theta_{[\ell_1,\ell_2]}^g(x,y)$ readily follows.
\end{proof}


\subsection{Hyperbolic cosine barrier} 
\label{sub:hyperbolic_cosine_barrier}

Observe now the key fact that, in our case $g(t)=r\ttsm\cosh(t)$, the probability appearing in \eqref{eq:dbmkernel} is reduced to the probability of a Brownian bridge staying below the linear function $2rt+\inv2r$, which can be computed explicitly.
In fact, assuming that $x\leq e^{-\ell_1}(2r \alpha+r/2)=r\ttsm\cosh(\ell_1)$ and $y\leq e^{-\ell_2}(2r \beta+r/2)=r\ttsm\cosh(\ell_2)$ (note that the probability below is obviously zero if either condition is not met), the Cameron-Martin-Girsanov formula yields
\begin{multline*}
\pp_{\substack{\hat b(\alpha)=e^{\ell_1}x\\\hat b(\beta)=e^{\ell_2}y}}\!\left(\hat b(t)\leq 2rt+\tfrac12r\ \text{on}\ [\alpha,\beta]\right)\\
  =1-e^{-r(e^{\ell_2}y-e^{\ell_1}x)+r^2(\beta-\alpha)}\frac{e^{-\frac{(e^{\ell_2}y-2r \beta-e^{\ell_1}x+2r \alpha)^2}{4(\beta- \alpha)}}}{e^{-\frac{(e^{\ell_2}y-e^{\ell_1}x)^2}{4(\beta- \alpha)}}}\,
 \pp_{\substack{\hat{b}(\alpha)=e^{\ell_1}x-2r\alpha\\\hat b(\beta)=e^{\ell_2}y-2r \beta}}\left(\max_{t\in[\alpha,\beta]}\hat b(t)>\tfrac12r\right).
\end{multline*}
The last probability can be computed easily using the reflection principle, and it equals $e^{-(e^{\ell_1}x-2r \alpha-r/2)(e^{\ell_2}y-2r\beta-r/2)/(\beta-\alpha)}$ (see for instance page 67 in \cite{handbookBM}).
Putting this back in our formula \eqref{eq:dbmkernel} for $\Theta_{[\ell_1,\ell_2]}^{g(t)}$ with $g(t)=r\ttsm\cosh(t)$, which for simplicity we will denote from now on as $\Theta^{(r)}_{[\ell_1,\ell_2]}$, gives
\begin{multline}\label{eq:thetarRpre}
\Theta^{(r)}_{[\ell_1,\ell_2]}(x,y)=\uno{x\leq r\ttsm\cosh(\ell_1),\,y\leq r\ttsm\cosh(\ell_2)}\,
e^{\frac12(y^2-x^2)+\ell_2}\,\frac{1}{\sqrt{4\pi(\beta-\alpha)}}\\
\times\left(e^{-(e^{\ell_1}x-e^{\ell_2}y)^2/(4(\beta-\alpha))}-e^{-r(e^{\ell_2}y-e^{\ell_1}x)+r^2(\beta-\alpha)-(e^{\ell_1}x+e^{\ell_2}y-2r(\alpha+\beta)-r)^2/(4(\beta-\alpha))}\right).
\end{multline}
The above expression splits into two terms.
Note that if we disregard the indicator function, then by the above derivation the first term corresponds to the solution of \eqref{eq:bound-pde} with $g=\infty$, and thus it is nothing but $e^{-(\ell_2-\ell_1){\sf D}}(x,y)$.
As a consequence, we deduce that
\begin{equation}\label{eq:thetarR}
\Theta^{(r)}_{[\ell_1,\ell_2]}=\bar{\sf P}_{r\ttsm\cosh(\ell_1)}\left(e^{-(\ell_2-\ell_1){\sf D}}-{\sf R}^{(r)}_{[\ell_1,\ell_2]}\right)\bar{\sf P}_{r\ttsm\cosh(\ell_2)},
\end{equation}
where $\bar{\sf P}_af(x)=({\sf I}-{\sf P}_a)f(x)=f(x)\uno{x\leq a}$ and ${\sf R}^{(r)}_{[\ell_1,\ell_2]}$ is the reflection term
\begin{equation}\label{eq:reflR}
{\sf R}^{(r)}_{[\ell_1,\ell_2]}(x,y)=\tfrac{1}{\sqrt{4\pi(\beta-\alpha)}}e^{\frac12(y^2-x^2)+\ell_2-r(e^{\ell_2}y-e^{\ell_1}x)+r^2(\beta-\alpha)-(e^{\ell_1}x+e^{\ell_2}y-2r(\alpha+\beta)-r)^2/(4(\beta-\alpha))}
\end{equation}
and, we recall, $\alpha=\frac14e^{2\ell_1}$, $\beta=\frac14e^{2\ell_2}$.

Now we set $-\ell_1=\ell_2=L$, so that by Proposition \ref{prop:dbmcont} we have
\begin{equation}
\pp\left(\lambda_N(t)\leq r\ttsm\cosh(t)~\forall\,t\in\rr\right)
=\lim_{L\to \infty}\det\!\left({\sf I}-\KGN+\Theta^{(r)}_{[-L,L]} e^{2L{\sf D}}\KGN\right).
\end{equation}
Using now the cyclic property of the Fredholm determinant and the identities $e^{2L{\sf D}}\KGN=(e^{L{\sf D}}\KGN)^2$ and $e^{-L{\sf D}}\KGN e^{L{\sf D}}\KGN=e^{L{\sf D}}\KGN e^{-L{\sf D}}\KGN=\KGN$ (which follow directly from \eqref{eq:etDKGN} and the orthonormality of the $\varphi_n$'s) we may rewrite the last identity as
\begin{equation}\label{eq:detRewr}
\pp\left(\lambda_N(t)\leq r\ttsm\cosh(t)~\forall\,t\in\rr\right)
=\lim_{L\to \infty}\det\!\left({\sf I}-\KGN+e^{L{\sf D}}\KGN\Theta^{(r)}_{[-L,L]} e^{L{\sf D}}\KGN\right).
\end{equation}

Note that $s\longmapsto\Theta^{(r)}_{[-L,s]}$ is a semigroup, so that $\Theta^{(r)}_{[-L,L]}=\Theta^{(r)}_{[-L,0]}\Theta^{(r)}_{[0,L]}$, and thus in view of \eqref{eq:thetarR} we may write
\begin{equation}
\Theta^{(r)}_{[-L,L]}=\bar{\sf P}_{r\ttsm\cosh(L)}\bigl(e^{-L{\sf D}}-{\sf R}^{(r)}_{[-L,0]}\bigr)\bar{\sf P}_{r}\bigl(e^{-L{\sf D}}-{\sf R}^{(r)}_{[0,L]}\bigr)\bar{\sf P}_{r\ttsm\cosh(L)}.
\end{equation}
Following \cite{cqr}, we decompose $\Theta^{(r)}_{[-L,L]}$ in the following way:
\begin{equation}\label{eq:thetaOmega}
\Theta^{(r)}_{[-L,L]}=\bigl(e^{-L{\sf D}}-{\sf R}^{(r)}_{[-L,0]}\bigr)\bar{\sf P}_{r}\bigl(e^{-L{\sf D}}-{\sf R}^{(r)}_{[0,L]}\bigr)-\Omega^{(r)}_L,
\end{equation}
where
\begin{multline}\label{eq:errkernel}
\Omega^{(r)}_L=\bigl(e^{-L{\sf D}}-{\sf R}^{(r)}_{[-L,0]}\bigr)\bar{\sf P}_{r}\bigl(e^{-L{\sf D}}-{\sf R}^{(r)}_{[0,L]}\bigr)\\
-\bar{\sf P}_{r\ttsm\cosh(L)}\bigl(e^{-L{\sf D}}-{\sf R}^{(r)}_{[-L,0]}\bigr)\bar{\sf P}_{r}\bigl(e^{-L{\sf D}}-{\sf R}^{(r)}_{[0,L]}\bigr)\bar{\sf P}_{r\ttsm\cosh(L)}.
\end{multline}
The idea is that $\Omega^{(r)}_L$ is an error term which goes to 0 as $L\to \infty$. This is the content of the next result, whose proof we defer to Appendix \ref{sec:pflemerr}:

\begin{lem}\label{lem:errkernel}
Assume $r>0$. Then $\wt\Omega^{(r)}_L:=e^{L{\sf D}}\KGN\Omega^{(r)}_Le^{L{\sf D}}\KGN\xrightarrow[L\to\infty]{}0$ in trace norm.
\end{lem}

Since the mapping ${\sf A}\longmapsto\det({\sf I}+{\sf A})$ is continuous with respect to the trace norm, the lemma together with \eqref{eq:detRewr} and \eqref{eq:thetaOmega} show that if
\begin{equation}\label{eq:lambda}
  \Lambda:=\lim_{L\to \infty}\Big[\KGN-e^{L{\sf D}}\KGN\bigl(e^{-L{\sf D}}-{\sf R}^{(r)}_{[-L,0]}\bigr)\bar{\sf P}_{r}\bigl(e^{-L{\sf D}}-{\sf R}^{(r)}_{[0,L]}\bigr)e^{L{\sf D}}\KGN\Big]
\end{equation}
exists in the trace class topology, then
\begin{equation}\label{eq:Llimitkernel}
\pp\left(\lambda_N(t)\leq r\ttsm\cosh(t)~\forall\,t\in\rr\right)=\det\!\big({\sf I}-\Lambda\big).
\end{equation}
But, as we will see next, the operator inside the brackets in \eqref{eq:lambda} in fact does not depend on $L$ (the analogous property was proved in \cite{cqr} in the setting of the Airy$_2$ process).
The key step is the following result:

\begin{lem}\label{lem:inkernel}
For all $L>0$,
\begin{equation}
e^{L{\sf D}}\KGN\,{\sf R}^{(r)}_{[-L,0]}=\KGN\,\varrho_r\ \qqand
{\sf R}^{(r)}_{[0,L]}e^{L{\sf D}}\KGN=\varrho_r\KGN,
\end{equation}
where $\varrho_r$ is the reflection operator $\varrho_rf(x)=f(2r-x)$.
\end{lem}

Using this lemma and the fact that $e^{L{\sf D}}\KGN\ts e^{-L{\sf D}}=\KGN$ we get
\begin{multline}
  \KGN-e^{L{\sf D}}\KGN\bigl(e^{-L{\sf D}}-{\sf R}^{(r)}_{[-L,0]}\bigr)\bar{\sf P}_{r}\bigl(e^{-L{\sf D}}-{\sf R}^{(r)}_{[0,L]}\bigr)e^{L{\sf D}}\KGN\\
  =\KGN-\KGN({\sf I}-\varrho_r)\bar{\sf P}_{r}({\sf I}-\varrho_r)\KGN=\KGN\varrho_r\KGN,
\end{multline}
where the second equality follows from the identity $({\sf I}-\varrho_r)\bar {\sf P}_r({\sf I}-\varrho_r)={\sf I}-\varrho_r$.
In view of \eqref{eq:lambda} and \eqref{eq:Llimitkernel}, this yields Proposition \ref{prop:dbm-formula}.
All that is left to prove then is Lemma \ref{lem:inkernel}.

\begin{proof}[Proof of Lemma \ref{lem:inkernel}]
We will only provide the proof of the first formula, the second one is very similar.
Using \eqref{eq:reflR} we write the kernel of the operator ${\sf R}^{(r)}_{[-L,0]}$ as
\begin{equation}
{\sf R}^{(r)}_{[-L,0]}(x,y)=\tfrac{1}{\sqrt{\pi(1-e^{-2L})}}e^{-ax^2+b_yx+c_y}
\end{equation}
with
\begin{equation}
a=\tfrac{1+e^{-2L}}{2(1-e^{-2L})},\qquad b_y=\tfrac{2e^{-L}(2r-y)}{1-e^{-2L}}\qqand c_y=-\tfrac{(1+e^{-2L})(2r-y)^2}{2(1-e^{-2L})}.
\end{equation}
This formula together with \eqref{eq:etDKGN} and the contour integral representation of the Hermite function $\varphi_n(x)$,
\begin{equation}\label{eq:integhermite}
\varphi_n(x)=(2^nn!\sqrt{\pi})^{-1/2}e^{-x^2/2}\frac{n!}{2\pi i}\oint dt \frac{e^{2tx-t^2}}{t^{n+1}}
\end{equation}
(where the contour of integration encircles the origin), gives us
\begin{multline}
e^{L{\sf D}}\KGN {\sf R}^{(r)}_{[-L,0]}(x,y)=\int_{\rr}dz\sum_{n=0}^{N-1}e^{Ln}\varphi_n(x)\varphi_n(z){\sf R}^{(r)}_{[-L,0]}(z,y)\\
=\tfrac{1}{\sqrt{\pi(1-e^{-2L})}}\sum_{n=0}^{N-1}e^{Ln}\varphi_n(x)(2^nn!\sqrt{\pi})^{-1/2}\frac{n!}{2\pi i}\oint dt\frac{e^{-t^2}}{t^{n+1}}\int_{\rr}dz\,e^{-z^2/2+2tz-az^2+b_yz+c_y}.
\end{multline}
The $z$ integral is just a Gaussian integral, and computing it the last expression becomes
\begin{multline}
\sum_{n=0}^{N-1}e^{Ln}\varphi_n(x)(2^nn!\sqrt{\pi})^{-1/2}\frac{n!}{2\pi i}\oint dt\frac{e^{-e^{-2L}t^2+2e^{-L}t(2r-y)-(2r-y)^2/2}}{t^{n+1}}\\
=\sum_{n=0}^{N-1}e^{Ln}\varphi_n(x)(2^nn!\sqrt{\pi})^{-1/2}\frac{n!}{2\pi i}\oint dt\frac{e^{-t^2+2t(2r-y)-(2r-y)^2/2}}{t^{n+1}e^{Ln}},
\end{multline}
where we have performed the change of variables $t\mapsto t\ts e^L$.
The last integral and its prefactor are nothing but $\varphi_n(2r-y)$, so this yields $e^{L{\sf D}}\KGN {\sf R}^{(r)}_{[-L,0]}(x,y)=\KGN(x,2r-y)$ as needed.
\end{proof}


\section{Connection with LOE} 
\label{sec:connection_with_loe}

This section is devoted to the proof of the following result:

\begin{prop}\label{prop:identity}
For $r\geq0$,
  \[\det\!\left({\sf I}-\KGN \varrho_r \KGN\right)_{L^2(\rr)}=F_{{\rm LOE},N}(2r^2).\]
\end{prop}

Together with Proposition \ref{prop:dbm-formula}, this proposition implies Theorem \ref{thm:dbmLOE}.

Let us start by introducing an explicit formula for $F_{{\rm LOE},N}$.
To that end, we will utilize the ensemble $\bar\lambda(1)<\bar\lambda(2)<\dotsm<\bar\lambda(N)$ obtained as the result of superimposing the eigenvalues of two independent copies of our LOE matrices, writing them in increasing order, and then keeping only the even labelled coordinates (i.e. keeping the largest, 3\textsuperscript{rd} largest, 5\textsuperscript{th} largest, and so on).
Observe that if $\lambda_{\rm LOE}(N)$ denotes the largest eigenvalue of an LOE matrix as in Section \ref{sub:goe_and_loe}, then
\begin{equation}\label{eq:LOELUE}
  \pp(\lambda_{\rm LOE}(N)\leq2r^2)^2=\pp(\bar \lambda(N) \leq 2r^2).  
\end{equation}
The advantage of this representation is that the superimposed ensemble $(\bar \lambda(i))_{i=1,\dotsc,N}$ is a determinantal process with a simple correlation kernel $\wtKLN$ (see \cite{forresterRains-decim}). The kernel $\wtKLN$ is given as follows. For $n\in\nn$, introduce the \emph{Laguerre function}
\begin{equation}\label{eq:def-psi}
  \psi_n(x)=e^{-x/2}L_n(x),
\end{equation}
where $L_n(x)$ is the $n$-th normalized Laguerre polynomial (so that $\|\psi_n\|_2=1$), and then define the \emph{Laguerre kernel} as  
\begin{equation}
\KLN(x,y)=\sum_{n=0}^{N-1}\psi_n(x)\psi_n(y).
\end{equation}
Then
\begin{equation}\label{eq:wtKLN}
\wtKLN(x,y)=-\frac{\partial}{\partial x}\int_0^y du\,\KLN(x,u).
\end{equation}
The determinantal structure of the superimposed ensemble leads directly to a formula for the distribution of $\bar \lambda(N)$ (see \cite{johansson} or \cite[(1.36)]{quastelRem-review}):
\begin{equation}\label{eq:ppLUE}
  \pp( \bar\lambda(N)\leq 2r^2)=\det\!\left({\sf I}-{\sf P}_{2r^2}\wtKLN {\sf P}_{2r^2}\right)_{L^2(\rr)}.
\end{equation}

Observe that $\wtKLN$ is a finite rank operator, and thus the last determinant can be represented as the determinant of a finite matrix.
More precisely, if we factor our operator as $\wtKGN={\sf K}_1{\sf K}_2$ with ${\sf K}_1\!:\ell^2(\{0,\dotsc,N-1\}) \longrightarrow L^2(\rr)$ and ${\sf K}_2\!:L^2(\rr)\longrightarrow \ell^2(\{0,\dotsc,N-1\})$ defined by the kernels
${\sf K}_1(x,n)=-\psi_n'(x)$ and ${\sf K}_2(n,y)=\int_0^ydu\,\psi_n(u)$, then the cyclic property of the Fredholm determinant implies that the determinant in \eqref{eq:ppLUE} equals $\det({\sf I}-{\sf K}_2{\sf K}_1)$, so that
\begin{multline}
\det\bigl({\sf I}-{\sf P}_{2r^2}\wtKLN {\sf P}_{2r^2}\bigr)_{L^2(\rr)}=\det\Bigl[\delta_{jk}+\int_{2r^2}^\infty dx\,\psi_j'(x)\int_{0}^x du\,\psi_k(u)\Bigr]_{j,k=0}^{N-1}\\
=\det\Bigl[\delta_{jk}-\int_{2r^2}^\infty dx\,\psi_j(x)\psi_k(x)-\psi_j(2r^2)\int_{0}^{2r^2} du\,\psi_k(u)\Bigr]_{j,k=0}^{N-1},
\end{multline}
where in the second equality we have integrated by parts.
Defining now a symmetric matrix\footnote{As a notational guide, note that while we have used sans-serif fonts to denote operators acting on a Hilbert space and their associated kernels, we are using regular fonts to denote (finite) matrices.} $L\in\rr^{N\times N}$ and two column vectors $R_1,R_2\in\rr^{N}$ by
\begin{equation}\label{eq:LR1R2}
L_{jk}=\int_{2r^2}^\infty dx\,\psi_j(x)\psi_k(x),\quad (R_1)_j=\psi_j(2r^2)\ \qand\ (R_2)_j=\int_{0}^{2r^2}\!\! du\,\psi_j(u),
\end{equation}
for $j,k\in\{0,\ldots,N-1\}$, we deduce by the last identity, \eqref{eq:LOELUE} and \eqref{eq:ppLUE}, that
\begin{equation}\label{eq:LOEsq}
F_{{\rm LOE},N}(2r^2)^2=\det(I-L-R_1\otimes R_2).
\end{equation}

Similarly, we have a version of \eqref{eq:dbm-formula} in terms of the determinant of a finite matrix (which can be obtained by conjugating the kernel inside the Fredholm determinant in \eqref{eq:dbm-formula} by the operator ${\sf G}\!:L^2(\rr)\longrightarrow \ell^2(\{0,\dotsc,N-1\})$ with kernel ${\sf G}(n,x)=\varphi_n(x)$):
\begin{equation}\label{eq:KGN-H}
\det\bigl({\sf I}-\KGN\ts\varrho_r\KGN\bigr)_{L^2(\rr)}=\det(I-H),
\end{equation}
where the symmetric matrix $H$ has entries given by
\begin{equation}\label{eq:H}
  H_{jk}=\int_{\rr}dx\,\varphi_j(x)\varphi_k(2r-x).
\end{equation}
Therefore, and in view of \eqref{eq:LOEsq}, we see that, in order to prove Proposition \ref{prop:identity}, we have to establish that
\begin{equation}\label{eq:finite-dbmloe}
\det(I-H)^2=\det(I-L-R_1\otimes R_2).
\end{equation}

At this point the main difficulty in proving \eqref{eq:finite-dbmloe} lies in the fact that the two sides of the identity are given in terms of objects related to two different families of orthogonal polynomials, which makes it hard to relate one to the other.
So the first step in our proof of the identity consists in replacing the matrix $H$ on the left-hand side by a matrix defined in terms of Laguerre polynomials.

To this end, let us introduce the following $N\times N$ (real) matrix $\wt H$:
\begin{equation}\label{eq:wtH}
\wt H_{ij}=(-1)^N\left(\psi_{i+j-N}(2r^2)-\psi_{i+j-N+1}(2r^2)\right)\quad\text{for}\ i,j\in\{0,\ldots,N-1\}. 
\end{equation}
Here $\psi_n$ is the Laguerre function introduced in \eqref{eq:def-psi} for $n\geq0$, while we set $\psi_n\equiv0$ for $n<0$. Note in particular that $\wt H$ is zero above the anti-diagonal (i.e. $\wt H_{ij}=0$ if $i+j<N-1$).
This matrix will play a key role in the proof. As we will see in the next lemma, $\wt H$ is conjugate to $H$, so that $\det(I-H)=\det(I-\wt H)$. Moreover, we will see that the matrices $L$ and $R_1\otimes R_2$ are also intimately related to $\wt H$. In order to state the lemma we introduce a matrix $Q\in\rr^{N\times N}$ and two column vectors $u,v\in\rr^{N}$ by
\begin{equation}\label{eq:uvQ}
\begin{gathered}
  u=(-1)^{N-1}\uno{},\qquad v_i=(-1)^i2\quad\text{ for }i=0,\dotsc,N-1,\\
Q_{ij}=\begin{cases}
0 &\mbox{for } i<j,\\
-2r &\mbox{for } i=j,\\
-4r &\mbox{for } i>j,
\end{cases}
\qquad i,j=0,\ldots,N-1
\end{gathered}
\end{equation}
(here $\uno{}$ denotes the constant vector with 1 in each entry).

Note that the matrices and vectors introduced in this section are always indexed by $\{0,\dotsc,N-1\}$, and they generally depend on $N$ and $r$; we have omitted this dependence from the notation for simplicity.

\begin{lem}\label{lem:propmat}
Let $H$, $\wt H$, $L$, $R_1$, $R_2$, $Q$, $u$ and $v$ be defined as in \eqref{eq:LR1R2}, \eqref{eq:H}, \eqref{eq:wtH} and \eqref{eq:uvQ}.
Then the following properties hold:
\begin{enumerate}[label=\textnormal{(\roman{enumi})}]
\item\label{itm:1} $\wt H$ is conjugate to $H$, i.e. there exists an invertible matrix $S\in\rr^{N\times N}$ such that $\wt H=SHS^{-1}$.
\item\label{itm:2} $\wt H^2=L$.
\item\label{itm:3} $R_1=\wt Hu$ and $R_2=(I-\wt H)v$.
\item\label{itm:4} $\tfrac{\p}{\p r}\wt H=Q\wt H$.
\item\label{itm:5} $\tfrac{\p}{\p r}(I+\wt H)^{-1}=(I-\wt H^2)^{-1}\wt HQ+(I-\wt H^2)^{-1}E(I+\wt H)^{-1}$, where $E=4r\wt Hu\otimes u$.
\end{enumerate}
\end{lem}

This lemma contains all the key identities which will be needed in the proof of \eqref{eq:finite-dbmloe}.
Let us thus postpone the proof of the lemma until the end of this section and proceed directly to the proof of the main result of this section.

\begin{proof}[Proof of Proposition \ref{prop:identity}]
As we already explained, all we need to do is prove \eqref{eq:finite-dbmloe}.
The structure of the proof is inspired in that of the proof of \eqref{eq:GOE} in \cite{ferrariSpohn}.
Note that both sides of \eqref{eq:finite-dbmloe} are zero if $r=0$ (this is equivalent to the fact that both sides of \eqref{eq:LOELUE} vanish when $r=0$, which is clear).
Therefore we will assume throughout this proof that $r>0$, which for similar reasons implies that both $\det(I-H)$ and $\det(I-L-R_1\otimes R_2)$ are strictly positive.

We start by using \ref{itm:2} and \ref{itm:3} in Lemma \ref{lem:propmat} to rewrite the determinant on the left-hand side of \eqref{eq:finite-dbmloe} as
\begin{align*}
\det(I-L-R_1\otimes R_2)&=\det\bigl(I-\wt H^2-\wt Huv^{\sf T}(I-\wt H)\bigr)\\
&=\det(I+\wt H)\det\bigl(I-(I+\wt H)^{-1}\wt H uv^{\sf T}\bigr)\det(I-\wt H)\\
&=\det(I-\wt H)\det(I+\wt H)\bigl(1-\langle u,(I+\wt H)^{-1}\wt Hv\rangle\bigr),
\end{align*}
where in the third equality we used the fact that $(I+\wt H)^{-1}\wt H uv^{\sf T}$ is rank one.
By Lemma \ref{lem:propmat}\ref{itm:1}, we have $\det(I-H)=\det(I-\wt H)$, and thus \eqref{eq:finite-dbmloe} will follow if we prove that
\begin{equation}\label{eq:finite-detH}
\det(I-\wt H)=\det(I+\wt H)\bigl(1-\langle u,(I+\wt H)^{-1}\wt Hv\rangle\bigr).
\end{equation}
Note that, by the discussion in the last paragraph, since $r>0$, the left-hand side and the two factors on the right-hand side are strictly positive.

Consider the second factor on the right-hand side of \eqref{eq:finite-detH}.
Since $\langle u,v\rangle=0$ if $N$ is even and $\langle u,v\rangle=2$ if $N$ is odd, we can write $1=\langle u,v\rangle+(-1)^{N}$, so that
\begin{equation}
1-\langle u,(I+\wt H)^{-1}\wt Hv\rangle=(-1)^N+\langle u,v\rangle-\langle u,(I+\wt H)^{-1}\wt Hv\rangle=(-1)^N+\langle u,(I+\wt H)^{-1}v\rangle.
\end{equation}
Taking now logarithm on both sides we see that \eqref{eq:finite-detH} is equivalent to 
\begin{equation}\label{eq:logDet}
\log\det(I-\wt H)=\log\det(I+\wt H)+\log\bigl((-1)^N+\langle u,(I+\wt H)^{-1}v\rangle\bigr).
\end{equation}
We will prove that the derivatives in $r$ of both sides are equal, that is,
\begin{equation}\label{eq:tracedetH}
-{\rm Tr}\Bigl((I-\wt H)^{-1}\tfrac{\p}{\p r}\wt H\Bigr)={\rm Tr}\Bigl((I+\wt H)^{-1}\tfrac{\p}{\p r}\wt H\Bigr)+\frac{\langle u,\tfrac{\p}{\p r}(I+\wt H)^{-1}v\rangle}{(-1)^N+\langle u,(I+\wt H)^{-1}v\rangle},
\end{equation}
where we used the fact that $\frac{\p}{\p r}\log(\det(A))={\rm Tr}\bigl(A^{-1}\frac{\p}{\p r}A\bigr)$ if $A$ is a square matrix depending smoothly on $r$.
As a consequence, the two sides of \eqref{eq:logDet} differ at most by a constant.
But, since $\wt H \longrightarrow 0$ as $r\to\infty$, both sides of \eqref{eq:logDet} go to 0 as $r \rightarrow \infty$, so the two sides are equal.
Therefore our proof will be ready once we show that \eqref{eq:tracedetH} holds.

Since $(I-\wt H)^{-1}+(I+\wt H)^{-1}=2(I-\wt H^2)^{-1}$, \eqref{eq:tracedetH} is equivalent to
\begin{equation}\label{eq:tracedetH2}
-2\ts {\rm Tr}\Bigl((I-\wt H^2)^{-1}\tfrac{\p}{\p r}\wt H\Bigr)\!\left[(-1)^N+\langle u,(I+\wt H)^{-1}v\rangle\right]=\langle u,\tfrac{\p}{\p r}(I+\wt H)^{-1}v\rangle.
\end{equation}
At this stage we use Lemma \ref{lem:propmat}\ref{itm:4} and then the cyclicity of the trace to obtain
\[\nonumber-2\ts {\rm Tr}\Bigl((I-\wt H^2)^{-1}\tfrac{\p}{\p r}\wt H\Bigr)=
-2\ts {\rm Tr}\Bigl((I-\wt H^2)^{-1}Q\wt H\Bigr)=-2\ts {\rm Tr}\Bigl(Q\wt H(I-\wt H^2)^{-1}\Bigr).\]
Now note that if $A$ is an $N\times N$ real symmetric matrix then ${\rm Tr}(QA)=-2r\sum_{i,j=0}^{N-1}A_{ij}=-2r\langle u,Au\rangle$, and thus
\begin{equation}\label{eq:twoTrace}
  4r\langle u, (I-\wt H^2)^{-1}\wt Hu\rangle
  =-2\ts {\rm Tr}\Bigl((I-\wt H^2)^{-1}\tfrac{\p}{\p r}\wt H\Bigr).
\end{equation}
On the other hand, on the right-hand side of \eqref{eq:tracedetH2} we may apply Lemma \ref{lem:propmat}\ref{itm:5} and use the simple identity $Qv=(-1)^N4ru$ to get
\begin{equation}\label{eq:diffdetH}
\begin{aligned}
\langle u,\tfrac{\p}{\p r}(I+\wt H)^{-1}v\rangle&=\langle u,(I-\wt H^2)^{-1}\wt HQv+4r(I-\wt H^2)^{-1}(\wt Hu\otimes u)(I+\wt H)^{-1}v\rangle\\
&=4r\langle u,(I-\wt H^2)^{-1}\wt Hu\rangle\!\left[(-1)^N+\langle u,(I+\wt H)^{-1}v\rangle \right].
\end{aligned}
\end{equation}
Using \eqref{eq:twoTrace} in \eqref{eq:diffdetH} we get \eqref{eq:tracedetH2}, which finishes the proof.
\end{proof}

\begin{proof}[Proof of Lemma \ref{lem:propmat}]
\mbox{}\\[4pt]
\ref{itm:1} Fix $N\in\nn$ and $r>0$, and define a upper triangular matrix $S\in\rr^{N\times N}$ as follows:
\begin{equation}\label{eq:defmatS}
S_{ij}=\mathtt c_j\tbinom{N-1-i}{j-i}(-1)^{N-1+j}\uno{j\geq i}\quad\text{with}\quad
\mathtt c_k=r^{N-1-k}\left(\tfrac{2^{N-1-k}k!}{(N-1)!}\right)^{1/2}
\end{equation}
for $i,j\in\{0,\ldots,N-1\}$.
We claim that $S$ is invertible, with inverse given by
\begin{equation}\label{eq:defmatSinv}
  S^{-1}_{ij}=\tfrac{1}{\mathtt c_i}\tbinom{N-1-i}{j-i}(-1)^{N-1+j}\uno{j\geq i}.
\end{equation}
To check this, note first that, since both $S$ and $S^{-1}$ (as given above) are upper triangular, we have $(SS^{-1})_{ij}=0$ for $i>j$, while $(SS^{-1})_{ii}=S_{ii}S^{-1}_{ii}=1$. Thus it remains to show that $(SS^{-1})_{ij}=0$ when $i<j$.
But
$(SS^{-1})_{ij}=\tfrac{(N-1-i)!}{(N-1-j)!}(-1)^{2N-2+j}\sum_{k=i}^j \tfrac{(-1)^k}{(k-i)!(j-k)!}
=\tfrac{(-1)^{i+j}(N-1-i)!}{(N-1-j)!(j-i)!}\sum_{k=0}^{j-i} (-1)^k\tfrac{(j-i)!}{k!(j-i-k)!}$,
and by the binomial theorem the last sum on the right-hand side is simply $(-1+1)^{j-i}=0$.

Now Lemma \ref{lem:intherm-sumlague} in Appendix \ref{sec:formula-hermlague} allows us to rewrite the symmetric matrix $H$ in terms of Laguerre functions: for $j\geq i$,
\begin{equation}\label{eq:idenmatHlague}
H_{ji}=H_{ij}=\frac{\mathtt c_j}{\mathtt c_i}\sum_{k=i}^j\binom{j-i}{k-i}(-1)^k\psi_k(2r^2).
\end{equation}
We will use this representation to show that $S^{-1}\wt HS=H$. We have
\begin{equation}
(S^{-1}\wt HS)_{ij}=\frac{\mathtt c_j}{\mathtt c_i}\sum_{\substack{k=i,\ldots, N-1\\\ell=0,\ldots,j}}\binom{N-1-i}{k-i}\binom{N-1-\ell}{j-\ell}(-1)^{2N-2+j+k}\wt H_{k\ell}.
\end{equation}
Note that the value of $\wt H_{k\ell}$ depends only on $k+\ell$.
Letting $\wt\psi_{n}=\psi_{n-1}(2r^2)-\psi_{n}(2r^2)$, so that $\wt H_{k\ell}=(-1)^N\wt\psi_{k+\ell-N+1}$, and recalling that, by convention, $\wt\psi_n=0$ for $n<0$, we may write
\begin{equation}\label{eq:sqBracketPre}
(S^{-1}\wt HS)_{ij}=\frac{\mathtt c_j}{\mathtt c_i}\sum_{n=0}^j\Biggl[\,\sum_{\substack{k=i,\dotsc,N-1,\;\ell=0,\dotsc,j\\k+\ell-N+1=n}}\binom{N-1-i}{k-i}\binom{N-1-\ell}{N-1-j}(-1)^{j+k+N}\Biggr]\wt\psi_{n}.
\end{equation}
Performing the change of variables $k\mapsto k+i$, $\ell\mapsto N-1-\ell$, and introducing the convention that $\binom{n}{m}=0$ if $m>n\geq0$, the sum in the square brackets turns into
\[\sum_{\substack{k\geq0,\;0\leq\ell\leq N-1\\k-\ell=n-i}}\binom{N-1-i}{k}\binom{\ell}{N-1-j}(-1)^{i+j+k+N}.\]
We claim now that the sum in $0\leq\ell\leq N-1$ can be extended to $\ell\geq0$. 
In fact, we may assume that $k\leq N-1-i$, since otherwise the first binomial coefficient vanishes. Since $\ell$ is constrained to be $\ell=k+i-n\leq N-1-n\leq N-1$, adding the terms with $\ell\geq N$ does not really contribute to the sum.
In view of this, and using Lemma \ref{lem:convbinom} from Appendix \ref{sec:formula-hermlague}, our sum can be rewritten as
\begin{equation}\label{eq:sqBracket}
\sum_{\substack{k\geq 0,\;\ell\geq 0\\\ell-k=i-n}}\binom{N-1-i}{k}\binom{\ell}{N-1-j}(-1)^{i+j+k+N}=
\begin{cases}
(-1)^{j+1}\binom{i-n}{i-j}\uno{i\geq j\geq n} &\mbox{for }i\geq n,\\
(-1)^{n+1}\binom{j-i-1}{n-i-1}\uno{j\geq n} &\mbox{for }i<n.
\end{cases}
\end{equation}
Now we substitute this formula into \eqref{eq:sqBracketPre} and consider three separate cases:

\begin{itemize}[leftmargin=0.8cm,itemsep=8pt]
  \item If $i=j$, then
  $(S^{-1}\wt HS)_{ii}=\sum_{n=0}^i\binom{i-n}{0}(-1)^{i+1}\wt\psi_{n}=(-1)^i\psi_{i}(2r^2)$.
  \item If $i<j$,  then
  $(S^{-1}\wt HS)_{ij}=\frac{\mathtt c_j}{\mathtt c_i}\sum_{n=i+1}^j(-1)^{n+1}\binom{j-i-1}{n-i-1}\wt\psi_n
    =\frac{\mathtt c_j}{\mathtt c_i}\sum_{n=i}^{j}\binom{j-i}{n-i}(-1)^n\psi_n(2r^2)$,
  where the second identity follows by summation by parts.
  \item If $i>j$, then proceeding as for $i<j$ we get
  $(S^{-1}\wt HS)_{ij}=
  \frac{\mathtt c_j}{\mathtt c_i}(-1)^j\sum_{n=0}^{j}\binom{i-n-1}{i-j-1}\psi_n(2r^2)$. 
  Applying Lemma \ref{lem:sumlague} from Appendix \ref{sec:formula-hermlague} we deduce that the last sum equals
  \begin{equation}
  \tfrac{\mathtt c_j}{\mathtt c_i}\tfrac{i!(2r^2)^j}{j!(2r^2)^i}\sum_{j\leq n\leq i}\tbinom{i-j}{n-j}(-1)^n\psi_n(2r^2)
  =\tfrac{\mathtt c_i}{\mathtt c_j}\sum_{j\leq n\leq i}\tbinom{i-j}{n-j}(-1)^n\psi_n(2r^2).
  \end{equation}
\end{itemize}
In each case, the expression for $(S^{-1}\wt HS)_{ij}$ coincides with the formula for $H_{ij}$ in \eqref{eq:idenmatHlague}, which completes the proof of \ref{itm:1}.\\[4pt]
\ref{itm:2} We will use the contour integral representation of the Laguerre function $\psi_n(x)$,
\begin{equation}\label{eq:integlague}
  \psi_n(x)=e^{-x/2}\frac{1}{2\pi\I}\oint dt\,\frac{e^{-\frac{xt}{1-t}}}{t^{n+1}(1-t)},
\end{equation}
where the integration is along a small circle around the origin (note that by Cauchy's theorem this formula is consistent with our convention $\psi_n\equiv0$ for $n<0$).
Together with the definition of $L$, \eqref{eq:integlague} leads to
\begin{align}\label{eq:ointL}
  L_{jk}=\int_{2r^2}^\infty dx\,\psi_j(x)\psi_k(x)
  &=\frac{1}{(2\pi\I)^2}\int_{2r^2}^\infty dx\oint\mspace{-9.0mu}\oint du\,dv\,\frac{e^{-x-\frac{xu}{1-u}-\frac{xv}{1-v}}}{u^{j+1}(1-u)v^{k+1}(1-v)}\\
  &=\frac{1}{(2\pi\I)^2}\oint\mspace{-9.0mu}\oint du\,dv\,\frac{e^{-2r^2\left(1+\frac{u}{1-u}+\frac{v}{1-v}\right)}}{u^{j+1}v^{k+1}(1-uv)}.
\end{align}
On the other hand, from the definition of $\wt H$ we get
\begin{align}\label{eq:ointH}
\nonumber&(\wt H^2)_{jk}=(-1)^{2N}\sum_{n=0}^{N-1}\left(\psi_{j+n-N}(2r^2)-\psi_{j+n-N+1}(2r^2)\right)\left(\psi_{n+k-N}(2r^2)-\psi_{n+k-N+1}(2r^2)\right)\\
\nonumber&=\frac{1}{(2\pi\I)^2}\oint\mspace{-9.0mu}\oint du\,dv\sum_{n=0}^{N-1}\frac{e^{-2r^2-\frac{2r^2u}{1-u}-\frac{2r^2v}{1-v}}}{(1-u)(1-v)}\Big(\tfrac{1}{u^{j+n-N+1}}-\tfrac{1}{u^{j+n-N+2}}\Big)\Big(\tfrac{1}{v^{n+k-N+1}}-\tfrac{1}{v^{n+k-N+2}}\Big)\\
&=\frac{1}{(2\pi\I)^2}\oint\mspace{-9.0mu}\oint du\,dv\,\frac{e^{-2r^2\left(1+\frac{u}{1-u}+\frac{v}{1-v}\right)}(1-(uv)^N)}{u^{j+1}v^{k+1}(1-uv)}.
\end{align}
The difference between \eqref{eq:ointL} and \eqref{eq:ointH} is then given by
\begin{equation}
L_{jk}-(\wt H^2)_{jk}=\frac{1}{(2\pi\I)^2}\oint\mspace{-9.0mu}\oint du\,dv\,\frac{e^{-2r^2\left(1+\frac{u}{1-u}+\frac{v}{1-v}\right)}(uv)^N}{u^{j+1}v^{k+1}(1-uv)}.
\end{equation}
Since $0\leq j,k\leq N-1$, the integrand has no poles in $u$ and $v$ inside the chosen contours, and hence the whole integral vanishes.\\[4pt]
\ref{itm:3} For the first formula, we compute directly $\wt Hu$ to get
\begin{equation}
(\wt Hu)_i=(-1)^{2N-1}\sum_{k=0}^{N-1}\left[\psi_{i+k-N}(2r^2)-\psi_{i+k-N+1}(2r^2)\right]
=\psi_i(2r^2)-\psi_{i-N}(2r^2)=(R_1)_i,
\end{equation}
where the last identity follows because $\psi_{i-N}(2r^2)=0$ (since $i<N$).
For the second one, we use the property $\frac{\p}{\p x}\!\left(L_n(x)-L_{n+1}(x)\right)=L_n(x)$ of Laguerre polynomials to obtain
$\frac{\p}{\p x}\left(\psi_n(x)-\psi_{n+1}(x)\right)=\frac{1}{2}\left(\psi_n(x)+\psi_{n+1}(x)\right)$, which, together with the fact that $L_n(0)=1$ for all $n\in\nn$, gives
\begin{equation}
\frac{1}{2}\int_{0}^{2r^2}\!dx\left[\psi_n(x)+\psi_{n+1}(x)\right]=\psi_n(2r^2)-\psi_{n+1}(2r^2)
\end{equation}
for all $n\in\nn$. Hence we can write the entries of $\wt H$ as
\begin{equation}\label{eq:wtH2}
\wt H_{ij}=\begin{cases}
0 &\mbox{for } i+j<N-1,\\
(-1)^{N+1}e^{-r^2} &\mbox{for } i+j=N-1,\\
\tfrac{(-1)^N}{2}\left(\Psi_{i+j-N}(2r^2)+\Psi_{i+j-N+1}(2r^2)\right) &\mbox{for } i+j>N-1,
\end{cases}
\end{equation}
with $\Psi_n(s)=\int_0^{s}\!dx\,\psi_n(x)$
(note that $\psi_0(x)=e^{-x/2}$). Now we can compute
\begin{align}
((I-\wt H)v)_i&=\sum_{k=0}^{N-1}(\delta_{ik}-\wt H_{ik})(-1)^k2
=2(-1)^i-2(-1)^{2N-i}e^{-r^2}-2\sum_{k=N-i}^{N-1}(-1)^k\wt H_{ik}\\
&=2(-1)^i(1-e^{-r^2})-2\sum_{k=0}^{i-1}(-1)^{k+N-i}\frac{(-1)^N}2(\Psi_k(2r^2)+\Psi_{k+1}(2r^2))\\
&=2(-1)^i(1-e^{-r^2})-(-1)^{i}\Psi_0(2r^2)+\Psi_i(2r^2)=(R_2)_i.
\end{align}
\\[4pt]
\ref{itm:4} From \eqref{eq:wtH2} we get
\begin{equation}
\tfrac{\p}{\p r}\wt H_{ij}=
\begin{cases}
0 &\mbox{for } i+j<N-1,\\
(-1)^{N}2r\ts e^{-r^2} &\mbox{for } i+j=N-1,\\
(-1)^N2r\!\left(\psi_{i+j-N}(2r^2)+\psi_{i+j-N+1}(2r^2)\right) &\mbox{for } i+j>N-1.
\end{cases}
\end{equation}
This expression coincides with the $i,j$ entry, for all $i,j\in\{0,\dotsc,N-1\}$, of the matrix $Q\wt H$, which is given by
\begin{equation}
(Q\wt H)_{ij}=-4r\sum_{k=0}^{i-1}\wt H_{kj}-2r\wt H_{ij}=(-1)^N 2r\!\left(\psi_{i+j-N}(2r^2)+\psi_{i+j-N+1}(2r^2)\right).
\end{equation}
\\[4pt]
\ref{itm:5} For $i,j\in\{0,\ldots,N-1\}$ we have
\begin{equation}
(Q\wt H)_{ij}+(\wt HQ)_{ij}=-4r\Bigl(\wt H_{ij}+\sum_{k=0}^{i-1}\wt H_{kj}+\sum_{k=j+1}^{N-1}\wt H_{ik}\Bigr).
\end{equation}
Since $\sum_{k=0}^{i-1}\wt H_{kj}=\sum_{k=0}^{j-1}\wt H_{ik}$, the right-hand side of the last identity equals $-4r\sum_{k=0}^{N-1}\wt H_{ik}$, which coincides with $-(4r\wt Hu\otimes u)_{ij}$.
Thus, recalling the notation $E=4r\wt Hu\otimes u$, we have
\begin{equation}\label{eq:hkpmatrices}
Q\wt H=-\wt HQ-E.
\end{equation}
Now recall that if $A$ is a square matrix which depends smoothly on a parameter $r$, then $\p_rA^{-1}=-A^{-1}\p_rAA^{-1}$.
Then, in view of \ref{itm:4} and the last identity, we have
\begin{equation}
  \p_r(I+\wt H)^{-1}=-(I+\wt H)^{-1}Q\wt H(I+\wt H)^{-1}
  =(I-\wt H^2)^{-1}(I-\wt H)(\wt HQ+E)(I+\wt H)^{-1}.
\end{equation}
Comparing this with the right-hand side of the identity we seek to prove, we see that it is enough to check that $\wt HQ(I+\wt H)+E=(I-\wt H)(\wt HQ+E)$,
which follows easily from \eqref{eq:hkpmatrices}.
\end{proof}


\appendix

\section{Some formulas for Hermite and Laguerre polynomials}\label{sec:formula-hermlague}

We begin with a combinatorial result which was used in the proof of Lemma \ref{lem:propmat}\ref{itm:1} and which will also be used later in this appendix. 

Throughout this appendix we adopt the convention that, for $k\in\nn$ and $\ell\in\zz$, $\binom{k}{\ell}=0$ if $\ell<0$ or $\ell>k$ (this can be justified, for instance, by replacing the factorials with Gamma functions).

\begin{lem}\label{lem:convbinom}
Let $n,m\in\nn$ and $a\in\zz$. Then
\begin{equation}
\sum_{\substack{i,j\geq 0\\j-i=a}}\binom{n}{i}\binom{j}{m}(-1)^i=
\begin{dcases*}
(-1)^n\binom{a}{m-n} & for $a\geq 0$,\\
(-1)^{m-a}\binom{n-m-1}{n-m+a}\uno{\{n\geq m-a\}} & for $a<0$.
\end{dcases*}
\end{equation}
\end{lem}

\begin{proof}
Assume first that $a\geq 0$. Then the formula we seek to prove can be rewritten as
\begin{equation}\label{eq:binombinom}
  \sum_{i\geq 0}\binom{n}{i}\binom{i+a}{m}(-1)^{n-i}=\binom{a}{m-n}.
\end{equation}
For $x\in\rr$ and with our convention, we have (using Newton's generalized binomial theorem)
\begin{multline}
\sum_{m\geq 0}\sum_{i\geq 0}\binom{n}{i}\binom{i+a}{m}(-1)^{n-i}x^m=\sum_{i\geq 0}\binom{n}{i}(-1)^{n-i}(1+x)^{i+a}\\
=(1+x)^a\left(-1+(1+x)\right)^{n}=(1+x)^a x^{n}=x^{n}\sum_{\ell\in\zz}\binom{a}{\ell}x^\ell.
\end{multline}
By equating the coefficient in front of $x^m$ on both sides, we obtain \eqref{eq:binombinom} .

When $a<0$, we first let $b=-a>0$ and rewrite the desired identity as
\begin{equation}\label{eq:caseneg}
\sum_{j\geq 0}\binom{n}{j+b}\binom{j}{m}(-1)^{j-m}=\binom{n-m-1}{n-m-b}\uno{\{n\geq m+b\}}.
\end{equation}
Pick $x\in\rr$ such that $|x|<1$ and $\big|\frac{x}{1-x}\big|<1$. Using three times the identity 
\begin{equation}\label{eq:binomialThing}
  \frac{x^k}{(1-x)^{k+1}}=\sum_{n\geq k}^\infty \binom{n}{k}x^n
\end{equation}
(which is a straightforward consequence of Newton's generalized binomial theorem) together with our convention we have
\begin{multline}
\sum_{n\geq 0}\sum_{j\geq 0}\binom{n}{j+b}\binom{j}{m}(-1)^{j-m}x^n=\sum_{j\geq 0}\binom{j}{m}\frac{x^{j+b}}{(1-x)^{j+b+1}}(-1)^{j-m}\\
=\frac{x^b}{(1-x)^{b+1}}\frac{\left(x/(1-x)\right)^m}{\left(1+x/(1-x)\right)^{m+1}}=\frac{x^{b+m}}{(1-x)^b}
=x^{m+1}\sum_{\ell\geq b-1}\binom{\ell}{b-1}x^\ell.
\end{multline}
\eqref{eq:caseneg} now follows from the fact that the coefficient of $x^n$ on the right-hand side is given by $\binom{n-m-1}{n-m-b}$ when $n\geq m+b$ and equals $0$ when $n<m+b$.
\end{proof}

\begin{lem}\label{lem:intherm-sumlague}
For $n,m\in\nn$ with $n\geq m$ and any $r\in\rr\setminus\{0\}$, the following relation holds:
\begin{equation}\label{eq:intherm-sumlague}
\int_\rr dx\,\varphi_n(x)\varphi_m(2r-x)=r^{m-n}\left(\frac{2^m n!}{2^n m!}\right)^{1/2}\sum_{k=m}^{n}\binom{n-m}{k-m}(-1)^k\psi_k(2r^2).
\end{equation}
Similarly, for the case $r=0$ we have
\begin{equation}\label{eq:intherm-sumlague-0}
\int_\rr dx\,\varphi_n(x)\varphi_m(-x)=(-1)^n\uno{m=n}.
\end{equation}
\end{lem}

\begin{proof}
Recall that the Hermite polynomials have a simple generating function, namely
\begin{equation}
\sum_{n=0}^\infty\frac{H_n(x)}{n!}t^n=e^{2xt-t^2}.
\end{equation}
We write the convolution of Hermite functions in \eqref{eq:intherm-sumlague} as
\begin{equation}
\int_\rr dx\,\varphi_n(x)\varphi_m(2r-x)=\frac{1}{\sqrt{2^{n+m}\pi n!m!}}\int_\rr dx\,H_n(x)e^{-x^2/2}H_m(2r-x)e^{-(2r-x)^2/2}
\end{equation}
and then use the above generating function to evaluate the sum
\begin{multline}\label{eq:gener-conv}
\sum_{n,m=0}^\infty\frac{t_1^nt_2^m}{n!m!}
\int_\rr dx\,H_n(x)e^{-x^2/2}H_m(2r-x)e^{-(2r-x)^2/2}\\
=\int_{\rr}dx\,e^{2xt_1-t_1^2-x^2/2+2(2r-x)t_2-t_2^2-(2r-x)^2/2}
=\sqrt\pi e^{-r^2+2r(t_1+t_2)-2t_1t_2}.
\end{multline}
By equating the coefficient of $t_1^nt_2^m$ on each side, we obtain an explicit formula for the left-hand side of \eqref{eq:intherm-sumlague}:
\begin{multline}\label{eq:intherm}
\int_\rr dx\,\varphi_n(x)\varphi_m(2r-x)=\frac{1}{\sqrt{2^{n+m}n!m!}}e^{-r^2}\p_{t_1}^n\p_{t_2}^m\left.e^{2r(t_1+t_2)-2t_1t_2}\right|_{t_1=t_2=0}\\
=\frac{1}{\sqrt{2^{n+m}n!m!}}e^{-r^2}\sum_{\ell=0}^m (-2)^\ell\ell!\binom{n}{\ell}\binom{m}{\ell}(2r)^{n+m-2\ell}.
\end{multline}
In particular, we get \eqref{eq:intherm-sumlague-0}, so from now on we will assume $r\neq0$.

Turning to  the right-hand side of \eqref{eq:intherm-sumlague}, we use the explicit power series expansion of the Laguerre polynomials,
\begin{equation}\label{eq:lagExp}
  L_k(x)=\sum_{\ell=0}^k\binom{k}{\ell}\frac{(-1)^\ell}{\ell!}x^\ell,
\end{equation}
to rewrite it as
\begin{equation}
r^{m-n}\!\left(\frac{2^m n!}{2^n m!}\right)^{1/2}\sum_{k=m}^{n}\sum_{\ell=0}^k\binom{n-m}{k-m}(-1)^ke^{-r^2}\binom{k}{\ell}\frac{(-1)^\ell}{\ell!}(2r^2)^\ell.
\end{equation}
We need to show that this expression equals the right-hand side of \eqref{eq:intherm} or, equivalently, that
\begin{equation}
  \sum_{k=0}^{n-m}\sum_{\ell=0}^{k+m}\binom{n-m}{k}\binom{k+m}{\ell}(-1)^{k+m}\frac{(-2r^2)^\ell}{\ell!}
  =\sum_{\ell=n-m}^n\binom{m}{n-\ell}(-1)^n\frac{(-2r^2)^\ell}{\ell!},
\end{equation}
where we have performed the changes of variables $k\mapsto k+m$ on the left-hand side and $\ell\mapsto n-\ell$ on the right-hand side.
Using our convention, this is equivalent to
\begin{equation}
  \sum_{\ell=0}^{\infty}\left[\sum_{k=0}^{\infty}\binom{n-m}{k}\binom{k+m}{\ell}(-1)^{n-m-k}\right]\!\frac{(-2r^2)^\ell}{\ell!}
  =\sum_{\ell=0}^\infty\binom{m}{n-\ell}\frac{(-2r^2)^\ell}{\ell!}.
\end{equation}
It follows from Lemma \ref{lem:convbinom} (or, more specifically, from \eqref{eq:binombinom}) that the coefficients of $\frac{(-2r^2)^\ell}{\ell!}$ on both sides of this identity coincide, and this finshes the proof.
\end{proof}

\begin{lem}\label{lem:sumlague}
For any $n,m\in\nn$ with $n\geq m$ and any $x\in\rr$, the following relation holds:
\begin{equation}\label{eq:sumlague}
\frac{(-1)^mx^n}{n!}\sum_{k=0}^{m}\binom{n-k-1}{n-m-1}L_k(x)
=\frac{x^m}{m!}\sum_{k=m}^{n}\binom{n-m}{k-m}(-1)^kL_k(x).
\end{equation}
\end{lem}

\begin{proof}
We will use \eqref{eq:lagExp} in order to extract the coefficients of $x^\ell$ in the polynomials appearing on both sides.
The coefficient of $x^\ell$ on the left-hand side of \eqref{eq:sumlague} is clearly 0 if $\ell<n$, while for $n\leq\ell\leq n+m$ it is given by
\begin{equation}\label{eq:lhscoef}
\frac{(-1)^{\ell+m-n}}{n!(\ell-n)!}\sum_{k=0}^{m}\binom{n-k-1}{n-m-1}\binom{k}{\ell-n}
=\frac{(-1)^{\ell+m-n}}{n!(\ell-n)!}\binom{n}{\ell-m},
\end{equation}
where we have used a variant of Vandermonde's identity which can be obtained by equating the coefficient of $x^{n}$ in the expansion of both sides of the identity $x\frac{x^{n-m-1}}{(1-x)^{n-m}}\frac{x^{\ell-n}}{(1-x)^{\ell-n+1}}=\frac{x^{\ell-m}}{(1-x)^{\ell-m+1}}$ obtained by using \eqref{eq:binomialThing}.
On the other hand, for $\ell<m$ the coefficient of $x^\ell$ on the right-hand side of \eqref{eq:sumlague} is clearly zero, while for $m\leq\ell\leq n+m$ it is given by
\begin{equation}\label{eq:rhscoef}
\frac{(-1)^{\ell-m}}{m!(\ell-m)!}\sum_{k=0}^{n-m}\binom{n-m}{k}\binom{k+m}{\ell-m}(-1)^{k+m}
=\frac{(-1)^{\ell+n-m}}{m!(\ell-m)!}\binom{m}{\ell-n},
\end{equation}
where we used the change of variables $k\mapsto k+m$ and Lemma \ref{lem:convbinom}.
Notice that \eqref{eq:rhscoef} equals $0$ by our convention if $m\leq\ell< n$ and it clearly equals \eqref{eq:lhscoef} if $n\leq \ell\leq n+m$ (recall that we are assuming $n\geq m$). The proof is thus complete.
\end{proof}

\section{Proof of Lemma \ref{lem:errkernel}} 
\label{sec:pflemerr}

Throughout the proof we will use $c_1$ and $c_2$ to denote positive constants which do not depend on $L$ and whose value may change from line to line.
We will denote by $\|\cdot\|_1$ and $\|\cdot\|_2$ the trace class and Hilbert-Schmidt norms of operators on $L^2(\rr)$. We recall that
\begin{equation}\label{eq:inenorms}
\|AB\|_1\leq\|A\|_2\|B\|_2\qqand \|A\|_2^2=\int dx\,dy\,A(x,y)^2
\end{equation}
if $A$ has integral kernel $A(x,y)$;
for more details see \cite[Section 2]{quastelRem-review} or \cite{simon}.

In view of \eqref{eq:errkernel} we write $\wt\Omega^{(r)}_L=\wt\Omega^{(r,1)}_L+\wt\Omega^{(r,2)}_L$, where
\begin{align}
\wt\Omega^{(r,1)}_L&=e^{L{\sf D}}\KGN {\sf P}_{r\ttsm\cosh(L)}\bigl(e^{-L{\sf D}}-{\sf R}^{(r)}_{[-L,0]}\bigr)\bar{\sf P}_{r}\bigl(e^{-L{\sf D}}-{\sf R}^{(r)}_{[0,L]}\bigr)\bar{\sf P}_{r\ttsm\cosh(L)}e^{L{\sf D}}\KGN,\\
\wt\Omega^{(r,2)}_L&=e^{L{\sf D}}\KGN \bigl(e^{-L{\sf D}}-{\sf R}^{(r)}_{[-L,0]}\bigr)\bar{\sf P}_{r}\bigl(e^{-L{\sf D}}-{\sf R}^{(r)}_{[0,L]}\bigr){\sf P}_{r\ttsm\cosh(L)}e^{L{\sf D}}\KGN.
\end{align}
We will focus on $\wt \Omega^{(r,1)}_L$ and show that it goes to zero in trace norm, the proof for $\wt\Omega^{(r,2)}_L$ is very similar so we will omit it.

We factor $\wt\Omega^{(r,1)}_L$ as
\begin{align}
  \wt\Omega^{(r,1)}_L=\Upsilon_1\Upsilon_2\qquad\text{with}\quad
\Upsilon_1&=e^{L{\sf D}}\KGN {\sf P}_{r\ttsm\cosh(L)}\bigl(e^{-L{\sf D}}-{\sf R}^{(r)}_{[-L,0]})\bar{\sf P}_{r}\\\qand
\Upsilon_2&=\bar{\sf P}_{r}\bigl(e^{-L{\sf D}}-{\sf R}^{(r)}_{[0,L]}\bigr)\bar{\sf P}_{r\ttsm\cosh(L)}e^{L{\sf D}}\KGN.
\end{align}
By \eqref{eq:inenorms}, it is enough to show that $\|\Upsilon_1\|_2\|\Upsilon_2\|_2 \longrightarrow0$ as $L\to \infty$.
We start with $\Upsilon_1$, which is made of two terms which we will bound separately.
By \eqref{eq:inenorms} and the fact that the family $(\varphi_n)_{n\in\nn}$ is orthonormal we have
\begin{align}
\|e^{L{\sf D}}&\KGN {\sf P}_{r\ttsm\cosh(L)}e^{-L{\sf D}}\bar{\sf P}_r\|_2^2
=\sum_{n,n'=0}^{N-1}\int_{-\infty}^\infty dx\int_{-\infty}^r dy\,e^{L(n+n')}\varphi_n(x)\varphi_{n'}(x)\\
&\hspace{1.85in}\times\int_{[r\ttsm\cosh(L),\infty)^2} dzdz'\,\varphi_n(z)\varphi_{n'}(z')e^{-L{\sf D}}(z,y)e^{-L{\sf D}}(z',y)\\
&=\sum_{n=0}^{N-1}e^{2nL}\int_{-\infty}^r dy\left(\int_{r\ttsm\cosh(L)}^\infty dz\,\varphi_n(z)e^{-L{\sf D}}(z,y)\right)^2\\
&\leq Ne^{2(N-1)L}\int_{r\ttsm\cosh(L)}^\infty dz\int_{-\infty}^\infty dy\,(e^{-L{\sf D}}(z,y))^2,
\end{align}
where we have used the Cauchy-Schwarz inequality.
Using the formula for the kernel of $e^{-L{\sf D}}$ which is implicit in \eqref{eq:thetarRpre} and \eqref{eq:thetarR} we see that that the $y$ integral is just a Gaussian integral, and computing it gives
\[\|e^{L{\sf D}}\KGN {\sf P}_{r\ttsm\cosh(L)}e^{-L{\sf D}}\bar{\sf P}_r\|_2^2\leq Ne^{2(N-1)L}\tfrac{\coth(L)-1}{\sqrt{4\pi\coth(L)}}\int_{r\ttsm\cosh(L)}^\infty dz\,e^{2L-z^2\tanh(L)}.\]
The last integral is bounded by $c_1\ts e^{2L-r^2\cosh(L)^2\tanh(L)}$ for all $L>0$, and thus, recalling that we are assuming $r>0$,
\begin{equation}
\|e^{L{\sf D}}\KGN {\sf P}_{r\ttsm\cosh(L)}e^{-L{\sf D}}\bar{\sf P}_r\|_2^2\leq c_1e^{2NL-c_2e^{2L}}.
\end{equation}
for sufficiently large $L$.
The estimate for the other term appearing in $\Upsilon_1$ is very similar and leads to the same type of bound.
We deduce that
\begin{equation}\label{eq:upsilon1}
\|\Upsilon_1\|_2\leq c_1e^{NL-c_2e^{2L}}
\end{equation}
for large enough $L$.
On the other hand it is easy to check that the same calculation as above leads to
\begin{equation}\label{eq:upsilon2}
\|\Upsilon_2\|_2\leq c_1e^{NL}
\end{equation}
(note that in this case the projection ${\sf P}_{r\ttsm\cosh(L)}$ appearing in $\Upsilon_1$ is replaced by $\bar{\sf P}_{r\ttsm\cosh(L)}$; this accounts for the fact that the factor $e^{-c_2e^{2L}}$ disappears from the upper bound).
By combining \eqref{eq:upsilon1} and \eqref{eq:upsilon2} together we immediately get
\begin{equation}
\|\wt\Omega^{(r,1)}_L\|_1\leq c_1\,e^{2NL-c_2e^{2L}}\xrightarrow[L\to \infty]{}0.
\end{equation}

\vs
\paragraph{\bf Acknowledgements}
The authors would like to thank Jeremy Quastel for valuable comments on this manuscript.
They also thank Gregory Schehr for pointing out to us formulas (102)--(103) in \cite{RS2}.
GBN and DR were partially supported by Conicyt Basal-CMM and by Programa Iniciativa Cient\'ifica Milenio grant number NC130062 through Nucleus Millenium Stochastic Models of Complex and Disordered Systems.
DR was also supported by Fondecyt Grant 1120309.

\printbibliography[heading=apa]

\end{document}